\documentclass[final,leqno]{siamltex}

\usepackage[pdftex]{graphics}
\usepackage{mathrsfs}
\usepackage{bbm}
\usepackage{amsmath}
\usepackage{amsfonts}
\usepackage{amssymb}
\usepackage{stmaryrd}
\usepackage{fancyheadings,subfigure}
\usepackage[small,normal,bf,up]{caption}
\usepackage{epsfig,graphicx,wrapfig,fancyvrb,listings,color,textcomp,float}
\usepackage{algorithmic,url,hyperref}
\usepackage[section]{algorithm}
\usepackage{array,multirow}
\usepackage{marvosym,ulem}

\newcommand{\norm}[1]{\left\Vert #1 \right\Vert} 
\renewcommand{\Re}{{\mathbbm R}}
\newcommand{\one}{1\hspace{-0,9ex}1} 
\newcommand{\mathsc}[1]{{\normalfont\textsc{#1}}}

\newcommand{\A}{\mathbf{A}}
\renewcommand{\b}{\mathbf{b}}

\newcommand{\f}{\mathfrak{f}}
\newcommand{\s}{\mathfrak{s}}

\newtheorem{remark}{Remark}
\newtheorem{example}{Example}

\def\tableofcontents{\section*{Contents.\@mkboth{CONTENTS}{CONTENTS}\hskip 1em}  
 \@starttoc{toc}}

\def\norm#1{\|#1\|}

\title{Multirate generalized additive Runge Kutta methods\thanks{
The work of A. Sandu has been supported in part by NSF through awards NSF
OCI--8670904397, NSF CCF--0916493, NSF DMS--0915047, NSF CMMI--1130667, 
NSF CCF--1218454, AFOSR FA9550--12--1--0293--DEF, AFOSR 12-2640-06,
and by the Computational Science Laboratory at Virginia Tech.}}

\author{Michael G\"unther\thanks{Bergische Universit\"at Wuppertal,
        Institute of Mathematical Modelling, Analysis and Compuational
        Mathematics (IMACM), Gau\ss strasse 20, D-42119 Wuppertal, Germany 
        ({\tt guenther@uni-wuppertal.de}).}
\and Adrian Sandu\thanks{Virginia Polytechnic Institute and State University, Computational Science Laboratory, Department of Computer Science, 2202 Kraft Drive, Blacksburg, VA 24060, USA ({\tt sandu@cs.vt.edu}).}
}

\begin{document}

\thispagestyle{empty}
\setcounter{page}{0}

\begin{Huge}
\begin{center}
Computational Science Laboratory Technical Report CSL-TR-{\tt 6/2013} \\
\today
\end{center}
\end{Huge}
\vfil
\begin{huge}
\begin{center}
Michael G\"{u}nther and Adrian Sandu
\end{center}
\end{huge}

\vfil
\begin{huge}
\begin{it}
\begin{center}
``{\tt Multirate generalized additive Runge Kutta methods}''
\end{center}
\end{it}
\end{huge}
\vfil

\begin{large}
\begin{center}
Computational Science Laboratory \\
Computer Science Department \\
Virginia Polytechnic Institute and State University \\
Blacksburg, VA 24060 \\
Phone: (540)-231-2193 \\
Fax: (540)-231-6075 \\ 
Email: \url{sandu@cs.vt.edu} \\
Web: \url{http://csl.cs.vt.edu}
\end{center}
\end{large}

\vspace*{1cm}

\begin{tabular}{ccc}
\includegraphics[width=2.5in]{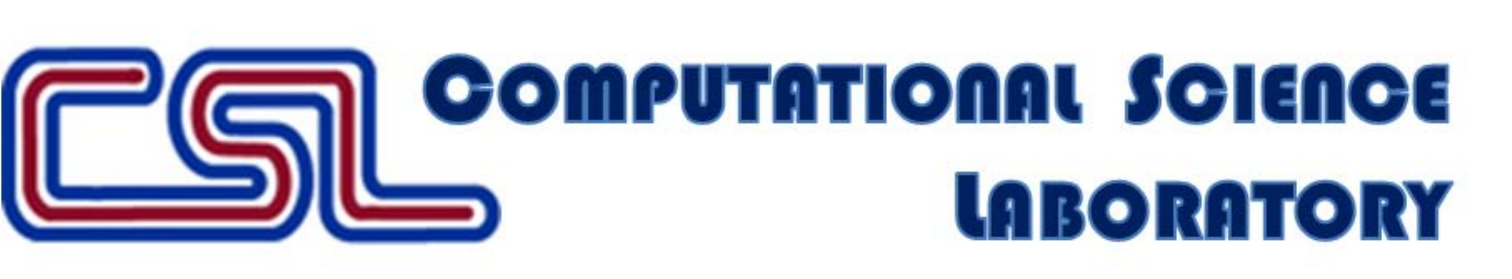}
&\hspace{2.5in}&
\includegraphics[width=2.5in]{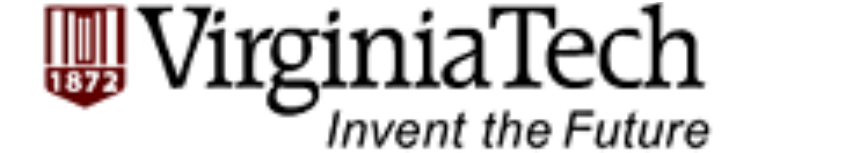} \\
{\bf\em Innovative Computational Solutions} &&\\
\end{tabular}

\newpage

\maketitle

\begin{abstract}
This work constructs a new class of multirate schemes based on the recently developed
generalized additive Runge-Kutta (GARK) methods~\cite{SaGu13a}. 
Multirate schemes use different step sizes 
for different components and for different partitions of the right-hand side based on the local activity levels. 
We show that the new multirate GARK family includes many well-known multirate schemes as special cases.
The order conditions theory follows directly from the GARK accuracy theory. Nonlinear stability and monotonicity investigations show that these properties are inherited from the base 
schemes provided that additional coupling conditions hold.
\end{abstract}



\pagestyle{myheadings}
\thispagestyle{plain}
\markboth{MICHAEL G\"UNTHER AND ADRIAN SANDU}{MULTIRATE GARK METHODS}

\section{Introduction}\label{sec:introduction}

Generalized additive Runge-Kutta (GARK) methods were introduced in ~\cite{SaGu13a}  to solve initial value problems for {\it additively} partitioned systems of ordinary differential equations 
\begin{equation} 
\label{eqn:additive-ode}
 y'= f(y) = \sum_{m=1}^N f^{\{m\}} (y), \quad y(t_0)=y_0\,,
\end{equation}
where the right-hand side $f: \Re^d \rightarrow \Re^d$ is split into in $m$ different parts 
with respect to, for example, stiffness, nonlinearity, dynamical behavior, and evaluation cost. Additive partitioning includes the case of {\em component} partitioning as follows. The set of indices $\{1,2,\ldots,d\}$ that number the solution components $y^i$ is split into $N$ subsets ${\cal I}^{\{m\}}$ to define 
\begin{equation}
\label{compwisepart}
f^{\{m\}}:= \sum_{i \in {\cal I}^{\{m\}}} e_i^T e_i\, f(y)\,, \quad \textnormal{i.e.,}\quad
f^{\{m\}i} (y) = \left\{
\begin{array}{ll}
f^i(y), \quad & i \in {\cal I}^{\{m\}}, \\
0, \quad & i \notin {\cal I}^{\{m\}}
\end{array}
\right. \,.
\end{equation}

A GARK method advances the numerical solution as follows ~\cite{SaGu13a} 
\begin{subequations}
\label{eqn:GARK}
\begin{eqnarray}
\label{eqn:delta-stage2}
Y_i^{\{q\}} &=& y_{n} + h \sum_{m=1}^N \sum_{j=1}^{s^{\{m\}}} a_{i,j}^{\{q,m\}} \, f^{\{m\}}(Y_j^{\{m\}}), \quad q=1,\ldots,N \, ,  \\
\label{eqn:delta-sol2}
y_{n+1} &=& y_n + h \sum_{q=1}^N\, \sum_{i=1}^{s^{\{q\}}} b_{i}^{\{q\}} \, f^{\{q\}}(Y_i^{\{q\}}) \,,
\end{eqnarray}
\end{subequations}
and is characterized by the extended Butcher tableau
\renewcommand{\arraystretch}{1.25}
\begin{equation}
\label{eqn:GARK-butcher}
\begin{array}{cccc}
\A^{\{1,1\}} & \A^{\{1,2\}} & \ldots & \A^{\{1,N\}} \\
\A^{\{2,1\}} &\A^{\{2,2\}} & \ldots & \A^{\{2,N\}} \\
\vdots & \vdots & & \vdots \\
\A^{\{N,1\}} & \A^{\{N,2\}} & \ldots & \A^{\{N,N\}} \\ \hline
\b^{\{1\}} & \b^{\{2\}} & \ldots &\b^{\{N\}}
\end{array}.
\end{equation}
\renewcommand{\arraystretch}{1.0}
Generalized additive Runge-Kutta methods show excellent stability properties and flexibility to exploit the different behavior of the partitions.
In contrast to additive Runge-Kutta schemes introduced in~\cite{KeCa2003}, GARK schemes 
allow for different stage values in the different partitions of $f$. Note that 
additive Runge-Kutta schemes are a special case of GARK
with ${\mathbf{A}}^{\{m,\ell\}}:={\mathbf{A}}^{\{\ell\}}$ for all $m,\ell=1,\ldots,N$.

This study develops new multirate schemes in the generalized additive Runge-Kutta framework. 
The paper is organized as follows. Section~\ref{sec:mgark-schemes} introduces the multirate GARK family and discusses their computational effort. 
The algebraic stability results for GARK schemes are transferred to multirate GARK schemes and order conditions are derived. 
Section~\ref{sec:mgark-traditional} shows that many existing multirate Runge-Kutta schemes can be represented and analyzed in the GARK framework. 
The generalization of additive Runge-Kutta schemes to multirate versions is considered in Section~\ref{sec:mark}. Section~\ref{sec:monotonicity} discusses absolutely monotonic multirate GARK schemes and shows how to construct such schemes. Finally,
conclusions are drawn in Section~\ref{sec:conclusions}.

\section{Generalized additive multirate schemes}\label{sec:mgark-schemes} 

\subsection{Formulation of generalized additive multirate schemes} 

We consider a two-way partitioned system \eqref{eqn:additive-ode} with one slow component $\{\mathfrak{s}\}$, and one active (fast) component $\{\mathfrak{f}\}$. The slow component is solved with a large step $H$, and the fast one with small steps $h=H/M$. Denote by $\widetilde{y}$ the intermediate solutions computed by the fast micro-steps, stating with $\widetilde{y}_n:= y_n$. A multirate generalization of \eqref{eqn:GARK} with $M$ micro steps $h=H/M$ proceeds as follows.
\begin{subequations}
\label{eqn:GARK-MR}
The slow stage values are given by: 
%
\begin{eqnarray}
\label{eqn:GARK-MR-slow-stage}
Y_i^{\{\s\}} & = & y_n + H \, \sum_{j=1}^{s^{\{\s\}}} a_{i,j}^{\{\s,\s\}} f^{\{\s\}}\left(Y_j^{\{\s\}}\right) + \\
 \nonumber
 & & 
                    + h \, \sum_{\lambda=1}^M \,\sum_{j=1}^{s^{\{\f\}}} a_{i,j}^{\{\s,\f,\lambda\}} f^{\{\f\}}\left(Y_j^{\{\f,\lambda\}}\right)\,, 
                     %
\qquad i=1,\dots,     s^{\{\s\}}\,.
\end{eqnarray}
The fast micro-steps are:
\begin{eqnarray}
\nonumber
\textnormal{For } \lambda=1,\ldots,M &&\\
\label{eqn:GARK-MR-fast-stage}
Y_i^{\{\f,\lambda\}} & = & \widetilde{y}_{n+\left(\lambda-1\right)/M} + H \, \sum_{j=1}^{s^{\{\s\}}} a_{i,j}^{\{\f,\s,\lambda\}} f^{\{\s\}}\left(Y_j^{\{\s\}}\right) + \\ 
\nonumber
& & +  h \, \sum_{j=1}^{s^{\{\f\}}} a_{i,j}^{\{\f,\f\}} f^{\{\f\}}\left(Y_j^{\{\f,\lambda\}}\right),~~ i=1,\dots,     s^{\{\f\}},\\
\widetilde{y}_{n+\lambda/M} & = & \widetilde{y}_{n+\left(\lambda-1\right)/M} + h \sum_{i=1}^{s^{\{\f\}}} b_{i}^{\{\f\}} f^{\{\f\}}\left(Y_i^{\{\f,\lambda\}}\right)\,.
\end{eqnarray}
The full (macro-step) solution is given by:
\begin{eqnarray}
y_{n+1} & = & \widetilde{y}_{n+M/M} + H \, \sum_{i=1}^{s^{\{\s\}}} b_{i}^{\{\s\}} f^{\{\s\}}\left(Y_i^{\{\s\}}\right) \,.
\end{eqnarray}
\end{subequations}

After eliminating the micro-step solutions $\widetilde{y}$  from the multirate GARK method~\eqref{eqn:GARK-MR} we arrive at the following form.

\begin{definition}[Multirate GARK method] One macro-step of a generalized additive multirate Runge-Kutta method with $M$ equal micro-steps reads
\begin{subequations}
\label{MGARK}
\begin{eqnarray}
\label{MGARK-slow-stage}
Y_i^{\{\s\}} & = & y_n + H \, \sum_{j=1}^{s^{\{\s\}}} a_{i,j}^{\{\s,\s\}} f^{\{\s\}}\left(Y_j^{\{\s\}}\right) + 
                  h \, \sum_{\lambda=1}^M \sum_{j=1}^{s^{\{\f\}}} a_{i,j}^{\{\s,\f,\lambda\}} f^{\{\f\}}\left(Y_j^{\{\f,\lambda\}}\right), \\
\quad Y_i^{\{\f,\lambda\}} & = & y_n + h \sum_{l=1}^{\lambda-1} \sum_{j=1}^{s^{\{\f\}}} b_{j}^{\{\f\}} f^{\{\f\}}\left(Y_j^{\{\f,l\}}\right) + 
\, \sum_{j=1}^{s^{\{\s\}}} a_{i,j}^{\{\f,\s,\lambda\}} f^{\{\s\}}\left(Y_j^{\{\s\}}\right) + \\ \nonumber
& & +  h \, \sum_{j=1}^{s^{\{\f\}}} a_{i,j}^{\{\f,\f\}} f^{\{\f\}}\left(Y_j^{\{\f,\lambda\}}\right), \quad \lambda=1,\ldots,M,\\
y_{n+1} & = & 
y_n + h \, \sum_{\lambda=1}^M \sum_{i=1}^{s^{\{\f\}}} b_{i}^{\{\f\}} f^{\{\f\}}\left(Y_i^{\{\f,\lambda\}}\right) + 
 H \, \sum_{j=1}^{s^{\{\s\}}} b_{i}^{\{\s\}} f^{\{\s\}}\left(Y_i^{\{\s\}}\right).
\end{eqnarray}
\end{subequations}
The base  schemes are Runge-Kutta methods, $(A^{\{\f,\f\}},b^{\{\f\}})$ for the slow component and 
$(A^{\{\s,\s\}},b^{\{\s\}})$ for the fast component. The coefficients $A^{\{\s,\f,\lambda\}}$, $A^{\{\f,\s,\lambda\}}$
realize the coupling between the two components.
\end{definition}

The method \eqref{MGARK} can be
written as a GARK scheme~\eqref{eqn:GARK}
over the macro-step $H$ with the fast stage vectors $Y^{\{\f\}}:=[Y^{\{\f,1\}}\,^T,\ldots,Y^{\{\f,M\}}\,^T]^T$. 
The corresponding Butcher tableau \eqref{eqn:GARK-butcher} reads
%
%
\begin{equation}
\label{eqn:mrRK-butcher}
\renewcommand{\arraystretch}{1.5}
\begin{array}{c|c}
{\mathbf{A}}^{\{\f,\f\}} & {\mathbf{A}}^{\{\f,\s\}}  \\ \hline
{\mathbf{A}}^{\{\s,\f\}} &{\mathbf{A}}^{\{\s,\s\}} \\ \hline 
{\mathbf{b}}^{\{\f\}} \,^T & {\mathbf{b}}^{\{\s\}} \,^T
\end{array} ~~ :=~~ 
\begin{array}{cccc|cccc}  
\frac{1}{M} A^{\{\f,\f\}}      &          0                   & \cdots & 0 & A^{\{\f,\s,1\}}  \\
\frac{1}{M} \mathbf{1} b^{\{\f\}}\,^T & \frac{1}{M} A^{\{\f,\f\}}        & \cdots & 0 &  A^{\{\f,\s,2\}}  \\
\vdots                     &                             & \ddots &   & \vdots  \\
\frac{1}{M} \mathbf{1} b^{\{\f\}}\,^T & \frac{1}{M} \mathbf{1} b^{\{\f\}}\,^T   & \ldots & \frac{1}{M} A^{\{\f,\f\}} &A^{\{\f,\s,M\}} \\
\hline 
\frac{1}{M} A^{\{\s,\f,1\}} & \frac{1}{M} A^{\{\s,\f,2\}} & \cdots & \frac{1}{M} A^{\{\s,\f,M\}} & A^{\{\s,\s\}}   \\   \hline 
\frac{1}{M} b^{\{\f\}}\,^T & \frac{1}{M} b^{\{\f\}}\,^T & \ldots & \frac{1}{M} b^{\{\f\}}\,^T & b^{\{\s\}}\,^T  
\end{array}
\renewcommand{\arraystretch}{1.0}
\end{equation}

\begin{example}[Simple MR GARK]
A simple version of \eqref{MGARK} uses the same coupling in all micro-steps,
\[
A^{\{\f,\s,1\}} =  \dots = A^{\{\f,\s,N\}} = A^{\{\f,\s\}}\,.
\]\
%
%
As we will see later, for stability reasons it is necessary in general to introduce additional freedom by using different coupling matrices for the micro-steps.
\end{example}

\begin{example}[Telescopic MR GARK]
Of particular interest are methods \eqref{MGARK} which use the same base scheme for both the slow and the fast components,
\begin{equation}
\label{eqn:mr-same-scheme}
A^{\{\f,\f\}} = A^{\{\s,\s\}} = A\,, \quad b^{\{\f\}} = b^{\{\s\}} = b\,.
\end{equation}
Such methods can be easily extended to systems with more than two scales by applying them in a telescopic fashion.
\end{example}

\subsection{Computational considerations}

The general formulation of the method \eqref{eqn:GARK-MR} leads to a system of coupled equations for all the fast and the slow stages, and the
resulting computational effort is larger, not smaller, than solving the coupled system with a small step. 
For an efficient computational process the macro and micro-steps need to stay mostly decoupled.

A very effective approach is to have the slow stages \eqref{eqn:GARK-MR-slow-stage} for $i=1,\dots,     s^{\{\s\}}$ coupled only with the first
fast micro-step,
\begin{eqnarray}
\label{eqn:GARK-MR-slow-stage-one-coupled}
Y_i^{\{\s\}} & = & y_n + H \, \sum_{j=1}^{s^{\{\s\}}} a_{i,j}^{\{\s,\s\}} f^{\{\s\}}\left(Y_j^{\{\s\}}\right) +
                      h \, \sum_{j=1}^{s^{\{\f\}}} a_{i,j}^{\{\s,\f\}} f^{\{\f\}}\left(Y_j^{\{\f,1\}}\right)\,. 
\end{eqnarray}
Equation \eqref{eqn:GARK-MR-slow-stage-one-coupled} and \eqref{eqn:GARK-MR-fast-stage} with $\lambda=1$ 
are solved together.
When both methods are implicit this first computation has a similar cost as one step of the coupled system.
The following fast micro-steps \eqref{eqn:GARK-MR-fast-stage} with $\lambda\ge 2$ are solved independently. The corresponding slow-fast coupling matrix is
\begin{equation}
\label{eqn:GARK-MR-slow-stage-one-coupled-A}
{\mathbf{A}}^{\{\s,\f\}}  = 
\begin{bmatrix}
\frac{1}{M} 
A^{\{\s,\f\}} & 0 & \cdots & 0  
\end{bmatrix}.
\end{equation}

When the slow stages are computed in succession, e.g., when the slow method is explicit or diagonally implicit, 
a more complex approach to decouple the computations is possible. Namely, the slow stages are coupled only with the micro-steps that have been computed already,
and vice-versa, the micro-steps are coupled only with the macro-stages whose solutions are already available. The fast and the slow methods proceed side by side. 
This decoupling can be achieved by choosing
\[
A^{\{\f,\s,\lambda\}} = \begin{bmatrix} \bar{A}^{\{\f,\s,\lambda\}} & \mathbf{0} \end{bmatrix} \in \Re^{s^{\{\f\}} \times s^{\{\s\}}}\,,
\quad \bar{A}^{\{\f,\s,\lambda\}} \in \Re^{s^{\{\f\}} \times j(\lambda)}\,,
\]
where $j(\lambda)$ is the number of slow stages that have been computed before the current micro-step, e.g.,
$c^{\{\s\}}_{j(\lambda)} \le (\lambda-1)/M$. 
The micro-step $\lambda$ is only coupled to these (known) stages.
Similarly,
\[
A^{\{\s,\f,\lambda\}} = \begin{bmatrix} \mathbf{0} \\ \bar{A}^{\{\s,\f,\lambda\}}  \end{bmatrix} \in \Re^{s^{\{\s\}} \times s^{\{\f\}}}\,,
\quad \bar{A}^{\{\s,\f,\lambda\}} \in \Re^{i(\lambda) \times s^{\{\f\}} }\,,
\]
where the first $s^{\{\s\}}-i(\lambda)$ slow stages are computed before the micro-step $\lambda$, and therefore
they are {\em not} coupled to the current micro-step.
An example of such methods is discussed in detail Section \ref{sec:MIS}.

\subsection{Nonlinear stability}
We consider systems \eqref{eqn:additive-ode} where each of the component functions  is dispersive:
\begin{equation}
\label{dispersive-condition}
\left\langle  f^{\{m\}} (y)- f^{\{m\}} (z)\,,\, y-z \right\rangle \le \nu^{\{m\}}\, \left\Vert y-z  \right\Vert^2\,, \quad \nu^{\{m\}} < 0 \,, \quad
 m \in \{\f,\s\}
\end{equation}
with respect to the same scalar product $\langle  \cdot, \cdot \rangle$.
For two solutions $y(t)$ and $\widetilde{y}(t)$  of \eqref{eqn:additive-ode}, each starting from a different initial condition,
the norm of the solution difference $\Delta y(t) = \widetilde{y}(t)- y(t)$ is non-increasing, 
$\lim_{\varepsilon>0, \varepsilon \to 0} \norm{\Delta y(t+\varepsilon)} \le \norm{\Delta y(t)}$.

This section investigates the conditions under which the multirate scheme \eqref{MGARK} is nonlinearly stable, i.e. the inequality
\[
 \| y_{n+1} - \tilde y_{n+1} \| \le \| y_{n} - \tilde y_{n} \|
\]
holds for any two numerical approximations $y_{n+1}$ and $\widetilde{y}_{n+1}$  obtained by applying the scheme to the ODE \eqref{eqn:additive-ode} with \eqref{dispersive-condition} and with initial values $y_n$ and $\widetilde{y}_n$.

\subsubsection{Additive partitioning}
Following the GARK stability analysis in~\cite{SaGu13a}, a multirate GARK scheme is algebraically stable if the following matrix is non-negative definite
\renewcommand{\arraystretch}{1.25}
\[
\mathbf{P} = \begin{bmatrix} \mathbf{P}^{\{\f,\f\}} & \mathbf{P}^{\{\f,\s\}} \\ \mathbf{P}^{\{\s,\f\}} & \mathbf{P}^{\{\s,\s\}}  \end{bmatrix} \ge 0\,,
\]
\renewcommand{\arraystretch}{1.0}
where
\begin{eqnarray*}
\mathbf{B}^{\{m\}} &:=& \mbox{diag}\left(\mathbf{b}^{\{m\}}\right)\,, \\
\mathbf{P}^{\{m,\ell\}} &:=& \mathbf{A}^{\{m,\ell\}}\,^T  \mathbf{B}^{\{m\}} + \mathbf{B}^{\{\ell\}} \mathbf{A}^{\{\ell,m\}} -  \mathbf{b}^{\{\ell\}} \mathbf{b}^{\{m\}}\,^T\,,  \qquad \forall~ m,\ell \in \{\f,\s\}\,.
\end{eqnarray*}
Algebraic stability guarantees unconditional nonlinear stability of the multirate GARK scheme~\cite{SaGu13a}. If the base schemes $\left(A^{\{\f,\f\}},b^{\{\f\}}\right)$ and $\left(A^{\{\s,\s\}},b^{\{\s\}}\right)$ are algebraically stable, one can easily verify that $\mathbf{P}^{\{\f,\f\}} \ge 0$ and $\mathbf{P}^{\{\s,\s\}} \ge 0$ hold, since
\begin{eqnarray*}
\mathbf{P}^{\{\s,\s\}} & = & {P}^{\{\s,\s\}}, \qquad
\mathbf{P}^{\{\f,\f\}}  =  \frac{1}{M^2}\, \mathbf{I}_{M \times M} \otimes  P^{\{\f,\f\}}\,. 
\end{eqnarray*}
%
%
The scheme \eqref{MGARK} is called {\em stability-decoupled} \cite{SaGu13a} if $\mathbf{P}^{\{\f,\s\}}=\mathbf{0}$ (and therefore  $\mathbf{P}^{\{\s,\f\}}=
\mathbf{P}^{\{\f,\s\}}\,^T=\mathbf{0}$). In this case
the individual stability of the slow and fast schemes is a sufficient condition for the stability of the overall multirate method. We have the following result.
\begin{theorem}[Stability of multirate GARK schemes]
Consider a  multirate GARK scheme~\eqref{MGARK} with positive fast weights,
$b^{\{\f\}}\,^i > 0$ for $i=1,\dots,s^{\{\f\}}$.
The scheme is stability-decoupled iff
$\mathbf{A}^{\{\f,\s\}}$ is given by 
\begin{eqnarray}
\label{eqn:aas-for-stability}
A^{\{\f,\s,\lambda\}} & := & B^{\{\f\}}\,^{-1} \left( {b}^{\{\f\}} {b}^{\{\s\}}\,^T -  A^{\{\s,\f,\lambda\}}\,^T B^{\{\s\}} \right)\,, \quad \lambda=1,\ldots,M.  
\end{eqnarray}
\end{theorem}
\begin{proof}
For the non-diagonal terms we have
\begin{eqnarray*}
\mathbf{P}^{\{\f,\s\}} & = & \mathbf{A}^{\{\f,\s\}}\,^T  \mathbf{B}^{\{\f\}} + \mathbf{B}^{\{\s\}} \mathbf{A}^{\{\s,\f\}} -  \mathbf{b}^{\{\s\}} \mathbf{b}^{\{\f\}}\,^T \\
& = & \frac{1}{M} 
\begin{bmatrix}
A^{\{\f,\s,1\}}\,^T B^{\{\f\}} &  A^{\{\f,\s,2\}}\,^T B^{\{\f\}} & \ldots & A^{\{\f,\s,M\}}\,^T B^{\{\f\}} 
\end{bmatrix}  \\
& & + \frac{1}{M}
\begin{bmatrix}
 B^{\{\s\}} A^{\{\s,\f,1\}}  &   B^{\{\s\}} A^{\{\s,\f,2\}} & \ldots & B^{\{\s\}} A^{\{\s,\f,M\}}
\end{bmatrix}  \\
& & 
- 
\frac{1}{M}
\begin{bmatrix}
{b}^{\{\s\}} {b}^{\{\f\}}\,^T &  {b}^{\{\s\}} {b}^{\{\f\}}\,^T
& \ldots & {b}^{\{\s\}} {b}^{\{\f\}}\,^T 
\end{bmatrix} \,.
\end{eqnarray*}
This term is zero if the following $M$  coupling conditions hold
\begin{eqnarray}
\label{eqn:coupling-for-stability}
A^{\{\f,\s,\lambda\}}\,^T B^{\{\f\}} + 
B^{\{\s\}} A^{\{\s,\f,\lambda\}} - 
 {b}^{\{\s\}} {b}^{\{\f\}}\,^T 
 & = & 0, 
 \quad \lambda=1,\ldots,M\,.
\end{eqnarray}
%
%
Solving for $A^{\{\f,\s,\lambda\}}\,^T$ yields~\eqref{eqn:aas-for-stability}. 
\qquad\end{proof}

\subsubsection{Component partitioning}
If we use component partitioning, no additional coupling conditions
have to be fulfilled as shown in~\cite{SaGu13a}, provided that both base schemes $(A^{\{\f,\f\}},b^{\{\f\}})$ and
$(A^{\{\s,\s\}},b^{\{\s\}})$ are algebraically stable.

\subsubsection{Conditional stability for coercive problems}

Following Higueras  \cite{Higueras_2005_monotonicity}, consider partitioned systems \eqref{eqn:additive-ode} where each of the component functions is coercive:
\begin{eqnarray}
\label{eqn:dispersive-condition-f}
\left\langle  f^{\{\s\}} (y)- f^{\{\s\}} (z)\,,\, y-z \right\rangle &\le& \mu\, \left\Vert f^{\{\s\}} (y)- f^{\{\s\}} (z) \right\Vert^2\,, \\
\nonumber
\left\langle  f^{\{\f\}} (y)- f^{\{\f\}} (z)\,,\, y-z \right\rangle &\le& \mu\, M\, \left\Vert f^{\{\f\}} (y)- f^{\{\f\}} (z) \right\Vert^2\,, \quad \mu < 0\,. 
\end{eqnarray}
Assume that there exist $r \ge 0$ such that the following matrix is positive definite
\renewcommand{\arraystretch}{1.25}
\begin{equation}
\label{eqn:conditional-stability-P}
\begin{bmatrix} \mathbf{P}^{\{\f,\f\}} + r\, M\, \mathbf{B}^{\{\f\}} & \mathbf{P}^{\{\f,\s\}} \\ \mathbf{P}^{\{\s,\f\}} &\mathbf{P}^{\{\s,\s\}}  +  r\, \mathbf{B}^{\{\s\}} \end{bmatrix}
\ge 0\,.
\end{equation}
\renewcommand{\arraystretch}{1.0}
The next result extends the one in \cite{SaGu13a}.

\begin{theorem}[Conditional stability of multirate GARK methods]
Consider a partitioned system \eqref{eqn:additive-ode} with coercive component functions
\eqref{eqn:dispersive-condition-f}
solved by a multirate GARK method, and assume that \eqref{eqn:conditional-stability-P} holds.
The solution is nonlinearly stable, in the sense that $\norm{\Delta y_{n+1}} \le \norm{\Delta y_{n}}$,
under the step size restriction 
\[
 H \le  \frac{-2\, \mu }{ r} \,.
\]
\end{theorem}
If the GARK method is stability decoupled then the weight $r$ in \eqref{eqn:conditional-stability-P} ensures
stability of the slow component for $H \le -2\mu/r$, and of the fast component under the step restriction $h \le -2\mu/(r M)$. 
The multirate GARK method imposes no additional step size restrictions for conditional stability.

\subsection{Order conditions}
As the multirate method \eqref{MGARK} is a particular instance of a  generalized additive Runge-Kutta scheme \eqref{eqn:GARK}, the order conditions follow directly from the derivation in~\cite{SaGu13a}. 
The order conditions for the multirate GARK methods~ \eqref{MGARK} 
are given in Tables~\ref{table_multGARK_a}
and \ref{table_multGARK_b}.
\begin{table}
\renewcommand{\arraystretch}{1.5}
\[
\begin{array}{c|rcl}
\mbox{$p$} & \multicolumn{3}{c}{\mbox{order condition}} \\ \hline
1 & b^{\{\s\}}\,^T \one & = & 1 \\ \hline
2 & b^{\{\s\}}\,^T A^{\{\s,\s\}} \one & = & \frac{1}{2} \\ 
  & b^{\{\s\}}\,^T \left(\sum_{\lambda=1}^M A^{\{\s,\f,\lambda\}} \right) \one & = & \frac{M}{2} \\ \hline
3 & b^{\{\s\}}\,^T \diag(A^{\{\s,\s\}} \one) A^{\{\s,\s\}} \one & =&  \frac{1}{3} \\
  & b^{\{\s\}}\,^T \diag(A^{\{\s,\s\}} \one) \left(\sum_{\lambda=1}^M A^{\{\s,\f,\lambda\}} \right) \one & =&  \frac{M}{3} \\
  & b^{\{\s\}}\,^T \diag(\sum_{\lambda=1}^M A^{\{\s,\f,\lambda\}} \one) A^{\{\s,\s\}} \one & = & \frac{M}{3} \\
  & b^{\{\s\}}\,^T \diag(\sum_{\lambda=1}^M A^{\{\s,\f,\lambda\}} \one) \left(\sum_{\lambda=1}^M A^{\{\s,\f,\lambda\}} \one  \right)& = & \frac{M^2}{3} \\
  & b^{\{\s\}}\,^T A^{\{\s,\s\}} A^{\{\s,\s\}} \one & = & \frac{1}{6} \\
 & b^{\{\s\}}\,^T A^{\{\s,\s\}} \left(\sum_{\lambda=1}^M A^{\{\s,\f,\lambda\}} \one  \right) & = & \frac{M}{6} \\
 & b^{\{\s\}}\,^T \left( \sum_{\lambda=1}^M  A^{\{\s,\f,\lambda\}} A^{\{\f,\s,\lambda\}} \right) \one & = & \frac{M}{6} \\
 & b^{\{\s\}}\,^T \left( 
\sum_{\lambda=1}^M A^{\{\s,\f,\lambda\}} \left\{
 A^{\{\f,\f\}} +(\lambda-1) I \right\} \one \right) & = &  \frac{M^2}{6} 
\end{array}
\]
\renewcommand{\arraystretch}{1.0}
\caption{\label{table_multGARK_a} Slow order conditions for multirate generalized Runge Kutta scheme~\eqref{MGARK}.}
\end{table}
\begin{table}
\[
\renewcommand{\arraystretch}{1.5}
\begin{array}{c|rcl}
\mbox{$p$} & \multicolumn{3}{c}{\mbox{order condition}} \\ \hline
1 & b^{\{\f\}}\,^T \one & = & 1 \\ \hline
2 & b^{\{\f\}}\,^T A^{\{\f,\f\}} \one & = & \frac{1}{2} \\ 
  & b^{\{\f\}}\,^T \left( \sum_{\lambda=1}^M A^{\{\f,\s,\lambda \}} \one \right)& = & \frac{M}{2} \\ \hline
3  & b^{\{\f\}}\,^T \diag(A^{\{\f,\f\}} \one) A^{\{\f,\f\}} \one & =&  \frac{1}{3} \\
  & b^{\{\f\}}\,^T \left( \diag(A^{\{\f,\f\}} \one) \sum_{\lambda=1}^M A^{\{\f,\s,\lambda\}} + \sum_{\lambda=1}^M (\lambda-1) A^{\{\f,\s,\lambda\}} \right) \one & =&  \frac{M^2}{3} \\
  & b^{\{\f\}}\,^T  \sum_{\lambda=1}^M \diag(A^{\{\f,\s,\lambda\}} \one) \left( A^{\{\f,\f\}} + (\lambda-1)  A^{\{\f,\s,\lambda\}} \right)  \one & =&  \frac{M^2}{3} \\
  & b^{\{\f\}}\,^T  \sum_{\lambda=1}^M  \diag(A^{\{\f,\s,\lambda\}} \one) A^{\{\f,\s,\lambda\}}  \one & = & \frac{M}{3} \\
  & b^{\{\f\}}\,^T A^{\{\f,\f\}} A^{\{\f,\f\}} \one & = & \frac{1}{6} \\
  & b^{\{\f\}}\,^T \left( A^{\{\f,\f\}} \sum_{\lambda=1}^M A^{\{\f,\s,\lambda\}} + \sum_{\mu=1}^{M-1} \sum_{\lambda=1}^\mu A^{\{\f,\s,\lambda\}} \right)\one & = & \frac{M^2}{6} \\
  & b^{\{\f\}}\,^T \left( \sum_{\lambda=1}^M A^{\{\f,\s,\lambda\}} \left\{ \sum_{\mu=1}^M A^{\{\s,\f,\mu\}} \right\} \right) \one & = & \frac{M^2}{6} \\
  & b^{\{\f\}}\,^T \left( \sum_{\lambda=1}^M A^{\{\f,\s,\lambda\}} \right) A^{\{\s,\s\}} \one & = &  \frac{M}{6}
\end{array}
\renewcommand{\arraystretch}{1.0}
\]
\caption{\label{table_multGARK_b} Fast order conditions for multirate generalized Runge Kutta scheme~\eqref{MGARK}.}
\end{table}
%
\subsubsection{Simplifying assumptions}
Consider the basis schemes $(A^{\{\f,\f\}},b^{\{\f\}})$ and $(A^{\{\s,\s\}},b^{\{\s\}})$ of order three or higher.
The multirate order conditions simplify considerably if the following conditions (named internal consistency conditions in \cite{SaGu13a}) hold
\begin{subequations}
\label{eqn:simplifying-condition-general}
\begin{eqnarray}
{\mathbf{A}}^{\{\f,\f\}} \one & = & {\mathbf{A}}^{\{\f,\s\}} \one:={\mathbf{c}^{\{\f\}}}, \\
{\mathbf{A}}^{\{\s,\s\}} \one & = & {\mathbf{A}}^{\{\s,\f\}} \one:={\mathbf{c}^{\{\s\}}},
\end{eqnarray}
\end{subequations}
or, in equivalent form
\begin{subequations}
\label{eqn:simplifying-condition}
\begin{eqnarray}
\label{cond.simpla}
\frac{1}{M} A^{\{\f,\f\}} \one + \frac{\lambda-1}{M} \one & = & A^{\{\f,\s,\lambda\}} \one= {\mathbf{c}^{\{\f,\lambda\}}}, 
\quad \lambda = 1,\dots,M\,,\\
\label{cond.simplb}
\frac{1}{M} \sum_{\lambda=1}^M A^{\{\s,\f,\lambda\}} \one & = &
 A^{\{\s,\s\}} \one = {\mathbf{c}^{\{\s\}}}.
\end{eqnarray}
\end{subequations}
If \eqref{eqn:simplifying-condition} hold then all order two conditions and most of the order three conditions are automatically fulfilled. The only remaining order three conditions are
\begin{subequations}
\label{remaining.order3-general}
\begin{eqnarray}
\left({\mathbf{b}}^{\{\s\}}\right)^T {\mathbf{A}}^{\{\s,\f\}} {\mathbf{c}^{\{\f\}}} & = & \frac{1}{6}, \\
\left({\mathbf{b}}^{\{\f\}}\right)^T {\mathbf{A}}^{\{\f,\s\}} {\mathbf{c}^{\{\s\}}} & = & \frac{1}{6},
\end{eqnarray}
\end{subequations}
or, in equivalent form,
\begin{subequations}
\label{remaining.order3}
\begin{eqnarray}
b^{\{\s\}}\,^T \sum_{\lambda=1}^M A^{\{\s,\f,\lambda\}} \left( A^{\{\f,\f\}} +(\lambda-1) I \right) \one & = & \frac{M^2}{6}, \\
\label{last.coupling}
b^{\{\f\}}\,^T \left( \sum_{\lambda=1}^M A^{\{\f,\s,\lambda\}} \right)
A^{\{\s,\s\}} \one & = & \frac{M}{6}\,.
\end{eqnarray}
\end{subequations}

When only the first fast microstep is coupled to the slow part \eqref{eqn:GARK-MR-slow-stage-one-coupled-A} the second simplifying condition~\eqref{cond.simplb} becomes
\begin{subequations}
\begin{equation}
\frac{1}{M} A^{\{\s,\f,1\}} \one = {\mathbf{c}^{\{\s\}}}; \quad A^{\{\s,\f,\lambda\}} = 0 \,,~~~ \lambda=2,\ldots,M. 
\end{equation}
The first simplifying condition~\eqref{cond.simpla} can be fulfilled by setting
\begin{equation}
\label{cond.simpl2}
A^{\{\f,\s,\lambda\}}  =  A^{\{\f,\s,1\}} + F(\lambda) \,,~~~ \lambda=1,\ldots,M
\end{equation}
with 
\begin{equation}
A^{\{\f,\s,1\}} \one = \frac{1}{M} A^{\{\f,\f\}} \one \quad \mbox{and} \quad F(\lambda) \one = \frac{\lambda-1}{M} \one,
\end{equation}
\end{subequations}
which transforms the last order three coupling condition~\eqref{last.coupling} into
\begin{eqnarray}
b^{\{\f\}}\,^T \left( M A^{\{\f,\s,1\}} +
\sum_{\lambda=1}^M F(\lambda) \right)  A^{\{\s,\s\}} \one & = & \frac{M}{6}.
\end{eqnarray}

\subsubsection{Additive partitioning}

We now consider the case of additive partitioning and impose the coupling conditions \eqref{eqn:coupling-for-stability} for stability. 
Following \eqref{eqn:aas-for-stability} we  set
\begin{eqnarray}
\label{add.part.stab}
\mathbf{A^{\{\f,\s\}}} := 
\begin{bmatrix}
\one {b}^{\{\s\}}\,^T - B^{\{\f\}}\,^{-1} A^{\{\s,\f,1\}}\,^T B^{\{\s\}}  \\
\vdots \\
\one {b}^{\{\s\}}\,^T - B^{\{\f\}}\,^{-1} A^{\{\s,\f,M\}}\,^T B^{\{\s\}} 
\end{bmatrix}.
\end{eqnarray}
If the base methods $(A^{\{\s,\s\}},b^{\{\s\}})$ and
$(A^{\{\f,\f\}},b^{\{\f\}})$ are algebraically stable then \eqref{add.part.stab} ensures the nonlinear stability of the overall method.
However, the stability conditions \eqref{add.part.stab} cannot be fulfilled when the
simplifying conditions~\eqref{eqn:simplifying-condition-general} hold.
\begin{theorem}[Internally consistent multirate GARK schemes are not stability decoupled]
When only the first fast microstep is coupled to the slow part \eqref{eqn:GARK-MR-slow-stage-one-coupled-A}, the  stability conditions~\eqref{add.part.stab} are not compatible with the first simplifying condition~\eqref{cond.simpla} for multirate GARK schemes with $M>1$.
\end{theorem}
\begin{proof}
Assume that the multirate GARK scheme fulfills the first
simplifying conditions \eqref{eqn:simplifying-condition} and the stability decoupling condition \eqref{add.part.stab}  for all $1 \le \lambda \le M$:
\begin{subequations}
\label{order3.gen.sc}
\begin{eqnarray}
\frac{1}{M} A^{\{\f,\f\}} \one + \frac{\lambda-1}{M} \one & = & A^{\{\f,\s,\lambda\}} \one, \\[3pt]
\one {b}^{\{\s\}}\,^T - B^{\{\f\}}\,^{-1} A^{\{\s,\f,\lambda\}}\,^T B^{\{\s\}} & = & {A^{\{\f,\s,\lambda\}}}.
\end{eqnarray}
\end{subequations}
Eliminating ${A^{\{\f,\s,\lambda\}}}$ leads to
\begin{subequations}
\label{order3.gen.sc2}
\begin{eqnarray}
\frac{1}{M} A^{\{\f,\f\}} \one + \frac{\lambda-1}{M} \one & = & \left( I
- B^{\{\f\}}\,^{-1} A^{\{\s,\f,\lambda\}}\,^T B^{\{\s\}} \right) \one, 
\label{order3.gen.sc2.1} 
\end{eqnarray}
\end{subequations}
When only the first fast microstep is coupled to the slow part \eqref{eqn:GARK-MR-slow-stage-one-coupled-A}, one gets for $M > 1$
$$
\frac{1}{M} A^{\{\f,\f\}} \one + \frac{\lambda-1}{M} \one  =   \one
\quad \forall \lambda \ge 2,
$$
yielding a contradiction.
\end{proof}

Consequently, if the base schemes are of order at least three  the
stability decoupling conditions~\eqref{add.part.stab}  require to work with the order conditions
given in Table~\ref{table_multGARK_noncomponent}. 
Note that condition pairs (i) and (viii),  (ii) and (xiv), (iii) and (xiii), (vi) and (xi), as well as (vii) and (x) coincide.
The coefficients of a stability decoupled multirate GARK scheme have to fulfill the following nine independent order conditions:
\[
\renewcommand{\arraystretch}{1.75}
\begin{array}{rrcl}
(iv) &  b^{\{\f\}}\,^T \left( \sum_{\lambda=1}^M \diag(D^\lambda \one) A^{\{\f,\f\}} + \sum_{\lambda=1}^M (\lambda-1) D^\lambda \right) \one & = & \frac{M^2}{6}, \\
(v) &  b^{\{\f\}}\,^T \left( \sum_{\lambda=1}^M A^{\{\f,\f\}} D^\lambda + \sum_{\mu=1}^{M-1} \sum_{\lambda=1}^\mu D^\lambda \right) \one & = & \frac{M^2}{3}, \\
(viii) & b^{\{\s\}}\,^T \left( \sum_{\lambda=1}^M A^{\{\s,\f,\lambda\}} \right) \one & = & \frac{M}{2}, \\ 
(ix) & b^{\{\s\}}\,^T \diag(A^{\{\s,\s\}} \one) \left( \sum_{\lambda=1}^M A^{\{\s,\f,\lambda\}} \right) \one & = & \frac{M}{3}, \\
(x) & b^{\{\s\}}\,^T \diag(\sum_{\lambda=1}^M A^{\{\s,\f,\lambda\}} \one) A^{\{\s,\s\}} \one & = & \frac{M}{3}, \\
(xi) & b^{\{\s\}}\,^T \diag(\sum_{\lambda=1}^M A^{\{\s,\f,\lambda\}} \one) \left( \sum_{\lambda=1}^M A^{\{\s,\f,\lambda\}} \right) \one & =&  \frac{M^2}{3}, \\
(xii) & b^{\{\s\}}\,^T A^{\{\s,\s\}} \left( \sum_{\lambda=1}^M A^{\{\s,\f,\lambda\}} \right)\one & = & \frac{M}{6}, \\
(xiii)   & b^{\{\s\}}\,^T \left( \sum_{\lambda=1}^M A^{\{\s,\f,\lambda\}} D^\lambda \right) \one & = & \frac{M}{3}, \\
(xiv)  & b^{\{\s\}}\,^T \left( \sum_{\lambda=1}^M A^{\{\s,\f,\lambda\}} \left( A^{\{\f,\f\}} +(\lambda-1) I \right) \right) \one & = &  \frac{M^2}{6}, 
\end{array}
\renewcommand{\arraystretch}{1.0}
\]
where
\begin{equation}
\label{eqn:D-lambda}
D^\lambda:= B^{\{\f\}}\,^{-1}\; A^{\{\s,\f,\lambda\}}\,^T\; B^{\{\s\}}\,.
\end{equation}

\begin{table}
\[
\renewcommand{\arraystretch}{1.5}
\begin{array}{ccrcl}
\mbox{No.} & \mbox{Order} & \multicolumn{3}{c}{\mbox{Order condition}} \\
(i) & 2  & b^{\{\f\}}\,^T \left( \sum_{\lambda=1}^M D^\lambda \right) \one & = & \frac{M}{2} \\ 
(ii) & 3 & b^{\{\f\}}\,^T \left( \diag(A^{\{\f,\f\}} \one) \left( \sum_{\lambda=1}^M D^\lambda \right) + \sum_{\lambda=1}^M (\lambda-1)D^\lambda \right)  \one & =&  \frac{M^2}{6} \\
(iii) & 3 & b^{\{\f\}}\,^T \left( \sum_{\lambda=1}^M \diag(D^\lambda \one) D^\lambda \right) \one & = & \frac{M}{3} \\
(iv) & 3 & b^{\{\f\}}\,^T \left( \sum_{\lambda=1}^M \diag(D^\lambda \one) A^{\{\f,\f\}} + \sum_{\lambda=1}^M (\lambda-1) D^\lambda \right) \one & = & \frac{M^2}{6} \\
(v) & 3 &  b^{\{\f\}}\,^T \left( \sum_{\lambda=1}^M A^{\{\f,\f\}} D^\lambda + \sum_{\mu=1}^{M-1} \sum_{\lambda=1}^\mu D^\lambda \right) \one & = & \frac{M^2}{3} \\
(vi) & 3 & b^{\{\f\}}\,^T \left( \sum_{\lambda=1}^M D^\lambda \left( \sum_{\mu=1}^M   A^{\{\s,\f,\mu\}} \right) \right) \one & = & \frac{M^2}{3} \\
(vii) & 3 & b^{\{\f\}}\,^T \left( \sum_{\lambda=1}^M D^\lambda \right) A^{\{\s,\s\}} \one & = & \frac{M}{3} \\ \hline 
(viii) & 2  & b^{\{\s\}}\,^T \left( \sum_{\lambda=1}^M A^{\{\s,\f,\lambda\}} \right) \one & = & \frac{M}{2} \\ 
(ix) & 3 & b^{\{\s\}}\,^T \diag(A^{\{\s,\s\}} \one) \left( \sum_{\lambda=1}^M A^{\{\s,\f,\lambda\}} \right) \one & = & \frac{M}{3} \\
(x) & 3 & b^{\{\s\}}\,^T \diag(\sum_{\lambda=1}^M A^{\{\s,\f,\lambda\}} \one) A^{\{\s,\s\}} \one & = & \frac{M}{3} \\
(xi) & 3 & b^{\{\s\}}\,^T \diag(\sum_{\lambda=1}^M A^{\{\s,\f,\lambda\}} \one) \left( \sum_{\lambda=1}^M A^{\{\s,\f,\lambda\}} \right) \one & =&  \frac{M^2}{3} \\
(xii)& 3 & b^{\{\s\}}\,^T A^{\{\s,\s\}} \left( \sum_{\lambda=1}^M A^{\{\s,\f,\lambda\}} \right)\one & = & \frac{M}{6} \\
(xiii) &3  & b^{\{\s\}}\,^T \left( \sum_{\lambda=1}^M A^{\{\s,\f,\lambda\}} D^\lambda \right) \one & = & \frac{M}{3} \\
(xiv) & 3 & b^{\{\s\}}\,^T \left( \sum_{\lambda=1}^M A^{\{\s,\f,\lambda\}} \left( A^{\{\f,\f\}} +(\lambda-1) I \right) \right) \one & = &  \frac{M^2}{6} 
\end{array}
\renewcommand{\arraystretch}{1.0}
\]
\caption{\label{table_multGARK_noncomponent} Order conditions (up to order three) for multirate generalized Runge Kutta scheme~\eqref{MGARK} with arbitrary partitioning. We use the definition \eqref{eqn:D-lambda} for $D^\lambda$.}
\end{table}

\begin{example}
Consider the case  \eqref{eqn:mr-same-scheme} where the same scheme is used for both the fast and the slow components. The order two coupling condition reads
\[
(viii)\quad  b^T\, \left( \sum_{\lambda=1}^M A^{\{\s,\f,\lambda\}} \right) \one  =  \frac{M}{2} 
\]
and the additional order three conditions are
\renewcommand{\arraystretch}{1.75}
\[
\begin{array}{rrcl}
(iv) &  b^T \left( \sum_{\lambda=1}^M \diag(B^{-1} A^{\{\s,\f,\lambda\}} b) A \one  + \sum_{\lambda=1}^M (\lambda-1) B^{-1} A^{\{\s,\f,\lambda\}} b \right)& = & \frac{M^2}{6} \\
(v) &  b^T \left( \sum_{\lambda=1}^M A B^{-1} A^{\{\s,\f,\lambda\}} + \sum_{\mu=1}^{M-1} \sum_{\lambda=1}^\mu B^{-1} A^{\{\s,\f,\lambda\}}  \right) b & = & \frac{M^2}{3} \\
(ix) & b^T \diag(A \one) \left( \sum_{\lambda=1}^M A^{\{\s,\f,\lambda\}} \right) \one & = & \frac{M}{3} \\
(x) & b^T \diag(\sum_{\lambda=1}^M A^{\{\s,\f,\lambda\}} \one) A \one & = & \frac{M}{3} \\
(xi) & b^T \diag(\sum_{\lambda=1}^M A^{\{\s,\f,\lambda\}} \one) \left( \sum_{\lambda=1}^M A^{\{\s,\f,\lambda\}} \right) \one & =&  \frac{M^2}{3} \\
(xii) & b^T A \left( \sum_{\lambda=1}^M A^{\{\s,\f,\lambda\}} \right)\one & = & \frac{M}{6} \\
(xiii)   & b^T \left( \sum_{\lambda=1}^M A^{\{\s,\f,\lambda\}} B^{-1} A^{\{\s,\f,\lambda\}} \right) b & = & \frac{M}{3} \\
(xiv)  & b^T \left( \sum_{\lambda=1}^M A^{\{\s,\f,\lambda\}} \left( A +(\lambda-1) I \right) \right) \one & = &  \frac{M^2}{6} 
\end{array}
\]
\renewcommand{\arraystretch}{1.0}
\renewcommand{\arraystretch}{1.0}

To easily construct schemes with an order higher than one, we set
now $A^{\{\s,\f,\lambda\}}=0$ for $\lambda=2,\ldots,M$.
For $s=2$ the only second order condition leads to
\renewcommand{\arraystretch}{1.25}
\[
A^{\{\s,\f,1\}} = \begin{bmatrix} 0 & 0 \\ \frac{M}{2 b_2} & 0  \end{bmatrix},
\quad
D^1 = \begin{bmatrix} 0 & \frac{M}{2 b_1} \\ 0 & 0  \end{bmatrix} =
\frac{b_2}{b_1}A^{\{\s,\f\}}\,^T,
\]
\renewcommand{\arraystretch}{1.0}
assuming $b_1 \neq 0$ and since $b_2 \neq 0$ for a method of order two. This choice preserves the explicit structure of the scheme, if the underlying scheme is explicit. 
Note that the basic method does not depend on $M$ -- only the coefficients of the coupling matrices $A^{\{\s,\f\}}$ and $D$ depend on $M$.
\end{example}
\subsubsection{Component partitioning}
%
For component partitioning the stability of the fast and slow base schemes ensures the overall stability, and no coupling conditions are needed. Consequently, 
there are additional degrees of freedom for choosing $\mathbf{A^{\{\f,\s\}}}$.
The general order conditions for this case have been given in
Tables~\ref{table_multGARK_a}
and \ref{table_multGARK_b}.

\begin{example}[First fast step coupling]
We construct a multirate scheme from two arbitrary 
order three basis schemes $(b^{\{\f\}},A^{\{\f,\f\}})$ and $(b^{\{\s\}},A^{\{\s,\s\}})$ 
by coupling only the first active microstep to the slow part and 
keeping the flexibility in the coupling matrices $A^{\{\f,\s,\lambda\}}$
The order two coupling conditions  are 
\begin{subequations}
\begin{eqnarray}
b^{\{\s\}}\,^T A^{\{\s,\f,1\}}  \one & = & \frac{M}{2}, \\
b^{\{\f\}}\,^T \left( \sum_{\lambda=1}^M A^{\{\f,\s,\lambda \}} \one \right)& = & \frac{M}{2}\,.
\end{eqnarray}
The order three conditions read
\begin{eqnarray}
b^{\{\s\}}\,^T \diag(A^{\{\s,\s\}} \one) A^{\{\s,\f,1\}}  \one & =&  \frac{M}{3}, \\
b^{\{\s\}}\,^T  \diag(A^{\{\s,\f,1\}} \one) A^{\{\s,\s\}} \one & = & \frac{M}{3}, \\
b^{\{\s\}}\,^T \diag(A^{\{\s,\f,1\}} \one) A^{\{\s,\f,1\}} \one  & = & \frac{M^2}{3}, \\
 b^{\{\s\}}\,^T A^{\{\s,\s\}} A^{\{\s,\f,1\}} \one   & = & \frac{M}{6}, \\
b^{\{\s\}}\,^T  A^{\{\s,\f,1\}} A^{\{\f,\s,1\}} \one & = & \frac{M}{6} , \\
b^{\{\s\}}\,^T A^{\{\s,\f,1\}}  A^{\{\f,\f\}} \one  & = &  \frac{M^2}{6}, \\
b^{\{\f\}}\,^T \left( \diag(A^{\{\f,\f\}} \one) \sum_{\lambda=1}^M A^{\{\f,\s,\lambda\}} + \sum_{\lambda=1}^M (\lambda-1) A^{\{\f,\s,\lambda\}} \right) \one & =&  \frac{M^2}{3}, \\
b^{\{\f\}}\,^T \left( \sum_{\lambda=1}^M \diag(A^{\{\f,\s,\lambda\}} \one) \left\{ A^{\{\f,\f\}} + (\lambda-1)  A^{\{\f,\s,\lambda\}} \right\} \right) \one & =&  \frac{M^2}{3}, \\
b^{\{\f\}}\,^T \left( \sum_{\lambda=1}^M \diag(A^{\{\f,\s,\lambda\}} \one) A^{\{\f,\s,\lambda\}} \right) \one & = & \frac{M}{3}, \\
b^{\{\f\}}\,^T \left( A^{\{\f,\f\}} \sum_{\lambda=1}^M A^{\{\f,\s,\lambda\}} + \sum_{\mu=1}^{M-1} \sum_{\lambda=1}^\mu A^{\{\f,\s,\lambda\}} \right)\one & = & \frac{M^2}{6}, \\
b^{\{\f\}}\,^T \left( \sum_{\lambda=1}^M A^{\{\f,\s,\lambda\}}\right)  A^{\{\s,\f,1\}} \one & = & \frac{M^2}{6}, \\
b^{\{\f\}}\,^T \left( \sum_{\lambda=1}^M A^{\{\f,\s,\lambda\}} \right) A^{\{\s,\s\}} \one & = &  \frac{M}{6}. 
\end{eqnarray}
\end{subequations}
The only degrees of freedom to fulfill these conditions are the parameters of the coupling coefficient matrices  $A^{\{\s,\f,1\}}$ and $A^{\{\f,\s,\lambda\}}$ ($\lambda=1,\ldots,M$).
\end{example}

\section{Traditional multirate Runge Kutta methods formulated in the GARK framework}\label{sec:mgark-traditional}

In this section we discuss several important multirate Runge Kutta methods proposed in the literature,
and show how they can be represented and analyzed in the GARK framework.

\subsection{Kvaerno-Rentrop methods}
The mRK class of multirate Runke-Kutta methods proposed by Kvaerno and Rentrop ~\cite{Kvaerno_1999_MR-RK}
can be formulated in the GARK framework. The mRK schemes are based on
coupling only the first fast microstep to the slow part, which, in this paper's notation, reads 
\[
A^{\{\s,\f\}}:=A^{\{\s,\f,1\}}, \quad
A^{\{\s,\f,\lambda\}}=0, \quad \lambda=2,\ldots,M.
\]
Kvaerno and Rentrop obtain order conditions that are nearly independent of $M$ by making the following choices of coefficients.
\begin{itemize}
\item[a)] The slow to fast coupling is
\begin{eqnarray*}
A^{\{\f,\s,\lambda+1\}} &:=& A^{\{\f,\s\}} + F(\lambda) \\
F(\lambda) &=& \one^{\{\f,\s\}}\, \begin{bmatrix}  \eta_1(\lambda) \dots \eta_{s^{\{\f,\s\}}}(\lambda) \end{bmatrix}\,,
\quad \lambda=0,\ldots,M-1\,,
\end{eqnarray*}
i.e., $F_{i,j}(\lambda)=\eta_j(\lambda)$. The scalar functions $\eta_j(\lambda)$ fulfill the condition 
\begin{equation}
\label{eqn:KR-simplifying}
\sum_{j=1}^{s^{\{\f,\s\}}} \eta_j(\lambda) = \lambda \quad \Leftrightarrow \quad
F(\lambda)\, \one= \lambda\, \one\,.
\end{equation}
\item[b)] The matrix $A^{\{\s,\f\}}$ is scaled by $M$, and the matrix $A^{\{\f,\s\}}$ is scaled by $1/M$, i.e., the function evaluations of the active part are always done in the microstep size, and the ones in the slow part in the macrostep size; 
\end{itemize}
Note that this choice corresponds to  the simplifying conditions~\eqref{eqn:simplifying-condition-general} with the special choice~\eqref{cond.simpl2}. 
With the notation 
\[
c^{\f}:=  A^{\{\f,\f\}} \one^{\f}\quad \textnormal{and} \quad 
c^{\s}:=  A^{\{\s,\s\}} \one^{\s}
\]
the corresponding GARK order conditions 
are given in Table~\ref{tab:mRK-order}.

Choose two order three schemes $(b^{\{\f\}},A^{\{\f,\f\}})$ and
$(b^{\{\s\}},A^{\{\s,\s\}})$. To obtain a multirate method of order three  the free parameters $A^{\{\f,\s\}}$, $A^{\{\s,\f\}}$, and $F(\lambda)$  
have to fulfill the two remaining order three coupling conditions, together with the three simplifying conditions,
\begin{subequations}
\label{eqn:coupling-conditions}
\begin{eqnarray}
b^{\{\s\}}\,^T A^{\{\s,\f\}} c^{\f} & = & \frac{M}{6},  \label{meth2.a} \\
b^{\{\f\}}\,^T \left(A^{\{\f,\s\}} + \frac{1}{M} \sum_{\lambda=0}^{M-1} F(\lambda) \right) c^{\s} & = & \frac{M}{6},   \label{meth2.b} \\
A^{\{\f,\s\}} \one & = & c^{\f},   \label{meth2.c} \\
A^{\{\s,\f\}} \one & = & c^{\s} ,  \label{meth2.d} \\
F(\lambda) \one & = & \lambda \label{meth2.e}\,,  \quad \lambda=0,\ldots,M-1.
\end{eqnarray}
\end{subequations}
Note that the additional condition 
\[
 b^{\{\f\}}\,^T F(\lambda) \, c^{\s}  =  \frac{\lambda (\lambda+1)}{2M}
\]
imposed by Kvaerno and Rentrop~\cite{Kvaerno_1999_MR-RK}, which transforms the second condition~\eqref{meth2.b} into
\[
b^{\{\f\}}\,^T A^{\{\f,\s\}} c^{\s}  =  \frac{1}{6M},
\]
ensures that the active solution has order three {\em at all microsteps.}

\begin{table}
\[
\renewcommand{\arraystretch}{1.5}
\begin{array}{c|rcl|rcl}
 \mbox{Order} & \multicolumn{3}{c|}{\mbox{Fast order condition}} & \multicolumn{3}{c}{\mbox{Slow order condition}} \\ \hline
 1 & b^{\{\f\}}\,^T \one & = & 1 & b^{\{\s\}}\,^T \one & = & 1 \\ \hline
 2 & b^{\{\f\}}\,^T c^{\f} & = & \frac{1}{2} & b^{\{\s\}}\,^T c^{\s} & = & \frac{1}{2} \\ 
 3 & b^{\{\f\}}\,^T \diag(c^{\f}) c^{\f} & =&  \frac{1}{3}  & b^{\{\s\}}\,^T \diag(c^{\s}) c^{\s} & =&  \frac{1}{3} \\
  & b^{\{\f\}}\,^T A^{\{\f,\f\}} c^{\f} & = & \frac{1}{6} & b^{\{\s\}}\,^T A^{\{\s,\s\}} c^{\s} & = & \frac{1}{6} \\
  & b^{\{\f\}}\,^T \left(A^{\{\f,\s\}} + \frac{1}{M} \sum_{\lambda=0}^{M-1} F(\lambda) \right) c^{\s} & = & \frac{M}{6}  & b^{\{\s\}}\,^T A^{\{\s,\f\}} c^{\f} & = & \frac{M}{6} \\
\end{array}
\renewcommand{\arraystretch}{1.0}
\]
\caption{\label{tab:mRK-order} Order conditions for the mRK~\cite{Kvaerno_1999_MR-RK} multirate GARK scheme with $F(\lambda) \one =\lambda \one$.}
\end{table}

\begin{example}[A multirate GARK schemes of order 3 with only two stages]
We are now interested in constructing schemes of order 3 with only 2 stages. 
Note that the overall scheme will be stable, if the basic schemes are stable due to componentwise partitioning. 
We use the simplifying conditions~\eqref{meth2.c} and~\eqref{meth2.d}
together with
\[
b:=b^{\{\f\}}=b^{\{\s\}},\quad 
A:=A^{\{\f\}}=A^{\{\s\}}, \quad 
c:=c^{\{\f\}}=c^{\{\s\}}\,, \quad 
\widetilde{A}:=A^{\{\f,\s\}} = A^{\{\s,\f\}}. 
\]
If we use an order three basis scheme $(b,A)$ , the remaining order conditions 
\eqref{eqn:coupling-conditions} for the free parameters $\widetilde{A}$ and $F(\lambda)$  read
\begin{subequations}
\label{eqn:coupling-conditions-particular}
\begin{eqnarray}
b^T \widetilde{A} \, c & = & \frac{M}{6},  \label{meth22.a} \\
b^T \left(\widetilde{A} + \frac{1}{M} \sum_{\lambda=0}^{M-1} F(\lambda) \right) c & = & \frac{M}{6},   \label{meth22.b} \\
\widetilde{A}\, \one & = & c ,  \label{meth22.c} \\
F(\lambda)\, \one & = & \lambda. \label{meth22.e}  
\end{eqnarray}
\end{subequations}
The first two conditions coincide if we set
\[
b^T \sum_{\lambda=0}^{M-1} F(\lambda)   \,
c = 0.
\]
When $c_{2} \neq c_{1}$ the choice
\[
\eta_1(\lambda) = \frac{c_{2}}{c_{2} - c_{1}} \lambda, \quad
\eta_2(\lambda) = - \frac{c_{1}}{c_{2} - c_{1}} \lambda,
\]
fulfills this additional condition and condition \eqref{meth22.e} at the same time. The remaining conditions~\eqref{meth22.a} and \eqref{meth22.c} can be fulfilled, for example, by setting
\[
\widetilde{A} = 
\begin{bmatrix}
c_1 & 0 \\
c_2-p & p
\end{bmatrix}
\quad \mbox{with}  \quad
p = \frac{\frac{M}{6}- b_1 c_1^2 - b_2 c_1 c_2}{b_2(c_2-c_1)}.
\]
This yields the order three GARK scheme given by the extended Butcher tableau
\[
\renewcommand{\arraystretch}{1.5}
\begin{array}{cccc|c}
\frac{1}{M} A     &          0                   & \cdots & 0 & \widetilde{A} \\
\frac{1}{M} \mathbf{1} b^T & \frac{1}{M} A        & \cdots & 0 & \widetilde{A} + F(1) \\
\vdots                     &                             & \ddots &   & \\
\frac{1}{M} \mathbf{1} b^T & \frac{1}{M} \mathbf{1} b^T   & \ldots & \frac{1}{M} A & \widetilde{A} + F(M-1)\\ \hline
\frac{1}{M} \widetilde{A} & 0 & \ldots & 0 & A \\ \hline
\frac{1}{M} b^T & \frac{1}{M} b^T & \frac{1}{M} b^T & \frac{1}{M} b^T & \frac{1}{M} b^T
\end{array}
\renewcommand{\arraystretch}{1.0}
\]

For the RADAU-IA scheme ($p=3,s=2$) we obtain
\[
\renewcommand{\arraystretch}{1.5}
\begin{array}{c|cr}
c & A \\ \hline & b^T \end{array}
:=
\begin{array}{c|cr}
0 & \frac{1}{4} & - \frac{1}{4} \\
\frac{2}{3} & \frac{1}{4} & \frac{5}{12} \\ \hline
& \frac{1}{4} & \frac{3}{4}
\end{array}\,;
\quad
F(\lambda) = \begin{bmatrix} \lambda & 0 \\ \lambda & 0 \end{bmatrix},
\quad \widetilde{A}= \begin{bmatrix} 0 & 0 \\ \frac{2}{3}-\frac{M}{3} & \frac{M}{3} 
\end{bmatrix}\,,
\renewcommand{\arraystretch}{1.0}
\]
and for RADAU-IIA ($p=3,s=2$) 
\[
\renewcommand{\arraystretch}{1.5}
\begin{array}{c|cr}
c & A \\ \hline & b^T \end{array}
:=
\begin{array}{c|cr}
0 & \frac{1}{4} & - \frac{1}{4} \\
\frac{2}{3} & \frac{1}{4} & \frac{5}{12} \\ \hline
& \frac{1}{4} & \frac{3}{4}
\end{array}\,;
\quad
F(\lambda) = \begin{bmatrix} \frac{3}{2} \lambda & - \frac{1}{2} \lambda \\
 \frac{3}{2} \lambda & - \frac{1}{2} \lambda \end{bmatrix},
\quad \widetilde{A}= \begin{bmatrix} \frac{1}{3} & 0 \\ 1-(M-1) & M-1 
\end{bmatrix}.
\renewcommand{\arraystretch}{1.0}
\]

\end{example}

\subsection{Dense output coupling}
The use of dense output interpolation for coupling the slow and fast components was developed by Savcenco, Hundsdorfer, and co-workers in the context of Rosenbrock methods \cite{Savcenco_2007TR,Savcenco_2008,Savcenco_2009,Savcenco_2005,Savcenco_2007_stability,Savcenco_2007_MRstrategy}. This approach can be immediately extended to Runge Kutta methods, and the overall scheme can be formulated in the mutirate GARK framework.

For a traditional Runge Kutta method the dense output provides highly accurate approximations of the solution at intermediate points
\begin{equation}
\label{eqn:dense-y}
y(t_{n}+\theta h)  \approx  
y_n 
+ H \, \sum_{j=1}^{s} b_{j}(\theta) f\left(Y_j\right)\,, \quad 0 \le \theta \le 1\,,
\end{equation}
or highly accurate approximations of the function values at intermediate points
\begin{equation}
\label{eqn:dense-f}
f\bigl( y(t_{n}+\theta h) \bigr) \approx  
\, \sum_{j=1}^{s} d_{j}(\theta) f\left(Y_j\right)\,, \quad 0 \le \theta \le 1\,.
\end{equation}

The slow terms  in the micro-steps \eqref{eqn:GARK-MR-fast-stage}
%
%
can be viewed as approximations of the function value at the micro steps
\[
 H f^{\{\s\}}\left( y(t_{n}+(\lambda-1+c_i^{\{\f,\f\}})h )\right) \approx
 H \, \sum_{j=1}^{s^{\{\s\}}} a_{i,j}^{\{\f,\s,\lambda\}} f^{\{\s\}}\left(Y_j^{\{\s\}}\right)\,.
\]
Consequently, using dense output of function values \eqref{eqn:dense-f} leads to the standard multirate GARK approach with the coupling given by the 
dense output coefficients
 \[
 a_{i,j}^{\{\f,\s,\lambda\}} = d_j\left(\frac{\lambda-1+c_i^{\{\f,\s,\lambda\}}}{M}\right) \,.
 \]
%
Alternatively, one can use the dense solution values \eqref{eqn:dense-y} in the micro-steps \eqref{eqn:GARK-MR-fast-stage}
\begin{eqnarray}
\label{eqn:microstep-dense-y}
Y_i^{\{\f,\lambda\}} & = & \widetilde{y}_{n+\left(\lambda-1\right)/M} + H \, f^{\{\s\}}\left(Y_i^{\{\s,\lambda\}}\right) + \\ 
\nonumber
& & +  h \, \sum_{j=1}^{s^{\{\f\}}} a_{i,j}^{\{\f,\f\}} f^{\{\f\}}\left(Y_j^{\{\f,\lambda\}}\right),~~ i=1,\dots,     s^{\{\f\}}
\end{eqnarray}
where
\begin{eqnarray*}
Y_i^{\{\s,\lambda\}} & = & 
y_n 
+ h \, \sum_{\lambda=1}^M \sum_{i=1}^{s^{\{\f\}}} b_{i}^{\{\f\}}(\lambda) f^{\{\f\}}\left(Y_i^{\{\f,\lambda\}}\right) 
+ H \, \sum_{i=1}^{s^{\{\s\}}} b_{i}^{\{\s\}}(\lambda) f^{\{\s\}}\left(Y_i^{\{\s\}}\right)\,.
\end{eqnarray*}
The dense output of the fast variable can be applied only for the current micro-step
\begin{eqnarray*}
Y_i^{\{\s,\lambda\}} & = & 
y_n 
+ H \sum_{i=1}^{s^{\{\f\}}} b_{i}^{\{\f\}}(\lambda) f^{\{\f\}}\left(Y_i^{\{\f,\lambda\}}\right) 
+ H \, \sum_{i=1}^{s^{\{\s\}}} b_{i}^{\{\s\}}(\lambda) f^{\{\s\}}\left(Y_i^{\{\s\}}\right)\,, 
\end{eqnarray*}
or for the previous micro-step, i.e., in extrapolation mode
\begin{eqnarray*}
Y_i^{\{\s,\lambda\}} & = & 
y_n 
+ H \sum_{i=1}^{s^{\{\f\}}} b_{i}^{\{\f\}}(\lambda) f^{\{\f\}}\left(Y_i^{\{\f,\lambda-1\}}\right) 
+ H \, \sum_{i=1}^{s^{\{\s\}}} b_{i}^{\{\s\}}(\lambda) f^{\{\s\}}\left(Y_i^{\{\s\}}\right)\,,
\end{eqnarray*}
where the dense output coefficients $b^{\{\s\}}(\lambda)$, $b^{\{\f\}}(\lambda)$ are appropriately redefined.

The solution interpolation approach \eqref{eqn:microstep-dense-y} can be cast in the GARK framework by
adding the additional slow stage values $Y_i^{\{\s,\lambda\}}$, with no contribution to the output ($b_i^{\{\s,\lambda\}}=0$). 
This is less convenient for analysis, however, as the number of slow stages becomes equal to the number of fast stages.


\subsection{Multirate infinitesimal step methods}\label{sec:MIS}

Multirate infinitesimal step (MIS) methods \cite{Wensch_2009_atmospheric} discretize the slow component with an explicit Runge Kutta method. The fast component is advanced between consecutive stages of this method
as the exact solution of a fast ODE system.
The fast ODE has a right hand side composed of the original fast component of the function, plus a piecewise constant ``tendency''  term representing the discretized slow component of the function. The order conditions of the overall method assume that the fast ODE can be solved exactly, which justifies the ``infinitesimal step'' name.
We show here that a multirate infinitesimal step method can be cast in the GARK framework when the inner fast ODEs are solved by a  Runge Kutta method with small steps.

We focus on the particular method of Knoth and Wolke \cite{Knoth_1998_MRimex}, which was the first MIS approach, and which has the
best practical potential. This approach has been named 
recursive flux splitting multirate (RFSMR) in \cite{Schlegel_2010_MR-imex,Schlegel_2011_MRimplementation,Schlegel_2012_multiscale},
and has been cast as a traditional partitioned Runge Kutta method in \cite{Schlegel_2012_multiscale}. 
Applications to the solution of atmospheric flows are discussed in \cite{Knoth_2010_MRimex,Schlegel_2009_MR-advection,Schlegel_2010_MR-imex,Schlegel_2011_MRimplementation}.
The approach below can be applied to any MIS scheme where the internal ODEs are solved by Runge Kutta methods.

Consider an outer (slow) explicit Runge Kutta scheme with the abscissae $c^\mathsc{o}_1=0$, $c^\mathsc{o}_i < c^\mathsc{o}_j$ for $i<j$, and $c^\mathsc{o}_s<1$.
The inner (fast) scheme can be explicit or implicit. If the same explicit scheme is used in both the inner and the outer loops then the method can be applied in a telescopic fashion,
where an even faster method is obtained by sub-stepping recursively. 

The scheme proceeds, in principle, by solving an ODE between each pair of consecutive stages of the slow explicit method:
\begin{equation}
\label{eqn:RFSMR-one-step}
\renewcommand{\arraystretch}{1.25}
\begin{array}{rcl}
\multicolumn{3}{l}{Y^{\{\s\}}_1 = y_n}  \\
\multicolumn{3}{l}{\textnormal{for } i=2,\dots,s^\mathsc{o}} \\
\quad v_i' &=& \sum_{j=1}^{i-1} \frac{a^\mathsc{o}_{i,j} - a^\mathsc{o}_{i-1,j}}{c^\mathsc{o}_{i} - c^\mathsc{o}_{i-1}} f^{\{\s\}} \left( Y^{\{\s\}}_j \right) + f^{\{\f\}} \left( v_i \right) \\
&& \textnormal{for  } \tau \in [0, (c^\mathsc{o}_{i} - c^\mathsc{o}_{i-1})\, H], \quad \textnormal{with } v_i(0) = Y^{\{\s\}}_{i-1} \\
\quad Y^{\{\s\}}_i &=& v_i\left( \tilde{c}^\mathsc{o}_i H \right) \\
\multicolumn{3}{l}{\textnormal{end for }i}  \\
\multicolumn{3}{l}{v' = \sum_{j=1}^{s^\mathsc{o}} \frac{b^\mathsc{o}_{j} - a^\mathsc{o}_{s^\mathsc{o},j}}{1 - c^\mathsc{o}_{s^\mathsc{o}}} f^{\{\s\}} \left( Y^{\{\s\}}_j \right) + f^{\{\f\}} \left( v \right)} \\
\multicolumn{3}{l}{\qquad  \textnormal{for  } \tau \in  [0, (1 - c^\mathsc{o}_{s^\mathsc{o}}) H], \quad \textnormal{with } v_i(0) = Y^{\{\s\}}_{s}}  \\
\multicolumn{3}{l}{y_{n+1} = v\left( 1 - c^\mathsc{o}_{s^\mathsc{o}} \right)}
\end{array}
\renewcommand{\arraystretch}{1.0}
\end{equation}
After a rescaling of the time variable $\theta = \tau/(c^\mathsc{o}_{i} - c^\mathsc{o}_{i-1})$ the internal ODEs read
\begin{subequations}
\label{eqn:RFSMR-internal-ode}
\begin{eqnarray}
\label{eqn:RFSMR-internal-ode-stage}
\quad v_i' &=& \sum_{j=1}^{i-1} \left( a^\mathsc{o}_{i,j} - a^\mathsc{o}_{i-1,j} \right)\, f^{\{\s\}}\left( Y^{\{\s\}}_j \right) +
\left(  c^\mathsc{o}_{i} - c^\mathsc{o}_{i-1} \right) \, f^{\{\f\}} \left( v_i \right) \,, \\
\nonumber
&&  \textnormal{for  } \theta \in  [0,  H], \quad \textnormal{with } v_i(0) = Y^{\{\s\}}_{i-1} \\
\nonumber
&& Y^{\{\s\}}_{i} = v_i(H) \\
\label{eqn:RFSMR-internal-ode-solution}
v' &=& \sum_{j=1}^{s^\mathsc{o}} \left(b^\mathsc{o}_{j} - a^\mathsc{o}_{s^\mathsc{o},j}\right)\, f^{\{\s\}} \left( Y^{\{\s\}}_j \right) + (1 - c^\mathsc{o}_{s^\mathsc{o}})\, f^{\{\f\}} \left( v \right) \\
\nonumber
&& \textnormal{for  } \theta \in  [0,  H], \quad \textnormal{with } v(0) = Y^{\{\s\}}_{s} \\
\nonumber
&& y_{n+1} = v(H)\,.
\end{eqnarray}
\end{subequations}

The numerical scheme solves the inner ODEs \eqref{eqn:RFSMR-internal-ode} using several steps of an inner Runge Kutta method \cite{Schlegel_2010_MR-imex}. For the present analysis 
we consider the case when only one step of the internal Runge Kutta method $(A^\mathsc{i},b^\mathsc{i})$ is taken to solve the ODE \eqref{eqn:RFSMR-internal-ode-stage} for each subinterval $i=2,\dots,s^\mathsc{o}$. This is no restriction of generality as any sequence of $M$ sub steps can be written as a single step method. The resulting scheme reads
\begin{subequations}
\label{eqn:RFSMR-one-internal-step}
\begin{eqnarray}
Y_{i,k}^{\{\f\}} & = & Y_{i-1}^{\{\s\}} + H\, c_k^\mathsc{i}\, \sum_{j=1}^{i-1} \left( a^\mathsc{o}_{i,j} - a^\mathsc{o}_{i-1,j} \right)\, f^{\{\s\}} \left( Y_j^{\{\s\}} \right) \\
\nonumber
&& + H\, \left(  c^\mathsc{o}_{i} - c^\mathsc{o}_{i-1} \right)\, \sum_{j=1}^{s^\mathsc{i}} a^\mathsc{i}_{k,j} \, f^{\{\f\}} \left( Y_{i,j}^{\{\f\}} \right)\,, \quad k=1,\dots,s^\mathsc{i}, \\ 
 Y_{i}^{\{\s\}} &=& Y_{i-1}^{\{\s\}} +  H\, \sum_{j=1}^{i-1} \left( a^\mathsc{o}_{i,j} - a^\mathsc{o}_{i-1,j} \right)\, f^{\{\s\}} \left( Y_j^{\{\s\}} \right) \\
\nonumber
&& + H\, \left(  c^\mathsc{o}_{i} - c^\mathsc{o}_{i-1} \right)\, \sum_{j=1}^{s^\mathsc{i}} b_{j}^\mathsc{i}\, f^{\{\f\}} \left( Y_{i,j}^{\{\f\}} \right).
\end{eqnarray}
\end{subequations}
The effective step of the inner method is $H\, \left(  c^\mathsc{o}_{i} - c^\mathsc{o}_{i-1} \right)$, which gives the multirate aspect of the method. If one performs $M$ sub steps then the effective step is $h\, \left(  c^\mathsc{o}_{i} - c^\mathsc{o}_{i-1} \right)$ with $h=H/M$.

Iterating after the explicit outer stages yields
\begin{eqnarray*}
Y_{i}^{\{\s\}} &=& y_n +  H\, \sum_{\ell=2}^{i} \, \sum_{j=1}^{\ell-1} \left( a^\mathsc{o}_{\ell,j} - a^\mathsc{o}_{\ell-1,j} \right)\, f^{\{\s\}} \left( Y_j^{\{\s\}} \right) 
\\
\nonumber
&& 
+ H\,  \sum_{\ell=2}^{i} \,\left(  c^\mathsc{o}_{\ell} - c^\mathsc{o}_{\ell-1} \right)\, \sum_{j=1}^{s^\mathsc{i}} b_{j}^\mathsc{i}\, f^{\{\f\}} \left( Y_{\ell,j}^{\{\f\}} \right)  \\
&=& y_n +  H \, \sum_{j=1}^{i-1} \, \sum_{\ell=j+1}^{i} \left( a^\mathsc{o}_{\ell,j} - a^\mathsc{o}_{\ell-1,j} \right)\, f^{\{\s\}} \left( Y_j^{\{\s\}} \right) 
\\
\nonumber
&& 
+ H\,  \sum_{\ell=2}^{i} \,\left(  c^\mathsc{o}_{\ell} - c^\mathsc{o}_{\ell-1} \right)\, \sum_{j=1}^{s^\mathsc{i}} b_{j}^\mathsc{i}\, f^{\{\f\}} \left( Y_{\ell,j}^{\{\f\}} \right) \\
&=& y_n +  H \, \sum_{j=1}^{i-1} \, a^\mathsc{o}_{i,j} \, f^{\{\s\}} \left( Y_j^{\{\s\}} \right)
+ H\,  \sum_{\ell=2}^{i} \,\left(  c^\mathsc{o}_{\ell} - c^\mathsc{o}_{\ell-1} \right)\, \sum_{j=1}^{s^\mathsc{i}} b_{j}^\mathsc{i}\, f^{\{\f\}} \left( Y_{\ell,j}^{\{\f\}} \right)\,.
\end{eqnarray*}
Equation \eqref{eqn:RFSMR-one-internal-step} becomes
%
\begin{eqnarray*}
\nonumber
Y_{i,k}^{\{\f\}} & = & y_n +  H\, \sum_{j=1}^{i-2}  a^\mathsc{o}_{i-1,j} \, f^{\{\s\}} \left( Y_j^{\{\s\}} \right) 
+ H\, c_k^\mathsc{i}\, \sum_{j=1}^{i-1} \left( a^\mathsc{o}_{i,j} - a^\mathsc{o}_{i-1,j} \right)\, f^{\{\s\}} \left( Y_j^{\{\s\}} \right) \\
&& + H\,  \sum_{\ell=2}^{i-1} \,\sum_{j=1}^{s^\mathsc{i}} b_{j}^\mathsc{i}\,\left(  c^\mathsc{o}_{\ell} - c^\mathsc{o}_{\ell-1} \right)\,  f^{\{\f\}} \left( Y_{\ell,j}^{\{\f\}}  \right)  \,,  \\ 
\nonumber
&& + H\,  \sum_{j=1}^{s^\mathsc{i}} \, \left(  c^\mathsc{o}_{i} - c^\mathsc{o}_{i-1} \right)\,  a_{k,j}^\mathsc{i}\, f^{\{\f\}} \left( Y_{\ell,j}^{\{\f\}}  \right) \, ,\quad k=1,\dots,s^\mathsc{i},  \\
 Y_{i}^{\{\s\}} &=& y_n +  H\, \sum_{j=1}^{i-1}  a^\mathsc{o}_{i,j} \, f^{\{\s\}} \left( Y_j^{\{\s\}} \right) 
+ H\,  \sum_{\ell=2}^{i} \, \sum_{j=1}^{s^\mathsc{i}} \, \left(  c^\mathsc{o}_{\ell} - c^\mathsc{o}_{\ell-1} \right)\,  b_{j}^\mathsc{i}\, f^{\{\f\}} \left( Y_{\ell,j}^{\{\f\}}  \right)  \,.
\end{eqnarray*}
The idea is now to interpret this scheme as a GARK method -- note that we formally solve the ODEs for the active part with one step of the inner fast method $(A^\mathsc{i},b^\mathsc{i})$ with $c^\mathsc{i} = A^\mathsc{i}\one$.
\begin{theorem}[The MIS scheme is a particular instance of a GARK method]
\begin{romannum}
\item
The MIS scheme~\eqref{eqn:RFSMR-one-step}--\eqref{eqn:RFSMR-one-internal-step}
can be written as a GARK method with the corresponding Butcher tableau \eqref{eqn:mrRK-butcher} given by $\mathbf{A}^{\{\s,\s\}}=A^\mathsc{o}$, $\mathbf{b}^{\{\s\}}=b^\mathsc{o}$,
\begin{eqnarray*}
\mathbf{A}^{\{\f,\f\}}  &=& 
\begin{bmatrix}
c_2^\mathsc{o}\, A^\mathsc{i}      &          0                   & \cdots & 0 \\
c_2^\mathsc{o}\,  \mathbf{1} b^\mathsc{i}\,^T & \left(c_3^\mathsc{o}-c_{2}^\mathsc{o}\right)\, A^\mathsc{i}        & \cdots & 0 \\
\vdots                     &                             & \ddots &   & \\
c_2^\mathsc{o}\,  \mathbf{1} b^\mathsc{i}\,^T & \left(c_3^\mathsc{o}-c_{2}^\mathsc{o}\right)\, \mathbf{1} b^\mathsc{i}\,^T   & \ldots & \left(1-c_{s^\mathsc{o}}^\mathsc{o}\right)\, A^\mathsc{i} 
\end{bmatrix} \in \Re^{s^\mathsc{o}s^\mathsc{i} \times s^\mathsc{o}s^\mathsc{i}} \,, \\
\mathbf{b}^{\{\f\}}\,^T & = &
\begin{bmatrix}
c_2^\mathsc{o}\,  b^\mathsc{i}\,^T & \left(c_3^\mathsc{o}-c_{2}^\mathsc{o}\right)\, b^\mathsc{i}\,^T & \ldots & \left(1-c_{s^\mathsc{o}}^\mathsc{o}\right) \, b^\mathsc{i}\,^T \end{bmatrix}
\in \Re^{s^\mathsc{o}s^\mathsc{i}} ,
\end{eqnarray*}
\begin{eqnarray*}
{\mathbf{A}}^{\{\f,\s\}} & = &
\begin{bmatrix}
 c^\mathsc{i}\,  \mathbf{e}_{2}^T\, A^\mathsc{o} \\
\vdots \\
\one\, \mathbf{e}_{i-1}^T\, A^\mathsc{o}  + c^\mathsc{i}\,  \left( \mathbf{e}_{i}^T-\mathbf{e}_{i-1}^T \right)\, A^\mathsc{o} \\
\vdots \\
\one\, \mathbf{e}_{s^\mathsc{o}}^T\, A^\mathsc{o}  + c^\mathsc{i}\,  \left( b^\mathsc{o}\,^T-\mathbf{e}_{s^\mathsc{o}}^T\, A^\mathsc{o}  \right)
\end{bmatrix} \in \Re^{s^\mathsc{o}s^\mathsc{i} \times s^\mathsc{o}} , 
\end{eqnarray*}
\begin{eqnarray*}
{\mathbf{A}}^{\{\s,\f\}} & = &
\begin{bmatrix}
c_2^\mathsc{o}\, \mathbf{g}_2 b^\mathsc{i}\,^T & 
\ldots &  \left(c_{s^\mathsc{o}}^\mathsc{o}-c_{s^\mathsc{o}-1}^\mathsc{o}\right)\, \mathbf{g}_{s^\mathsc{o}} b^\mathsc{i}\,^T & \mathbf{0} \end{bmatrix}  \in \Re^{s^\mathsc{o} \times s^\mathsc{o}s^\mathsc{i}} 
\end{eqnarray*}
where
\[
\mathbf{e}_i = \begin{bmatrix} 0 \\ \vdots \\ 1 \\ \vdots \\ 0 \end{bmatrix} \in \Re^{s^\mathsc{o}}\,, \quad
\mathbf{g}_i = \begin{bmatrix} 0 \\ \vdots \\ 1 \\ \vdots \\ 1 \end{bmatrix} \in \Re^{s^\mathsc{o}}
\]
\item
The coefficients fulfill the
 simplifying ``internal consistency'' conditions \eqref{eqn:simplifying-condition-general} given  in matrix form by
\[
\mathbf{c}^{\{\s,\s\}} = \mathbf{c}^{\{\s,\f\}} = \mathbf{c}^{\{\s\}} = c^\mathsc{o}\,,
\]
and
\[
\mathbf{c}^{\{\f,\s\}} = \mathbf{c}^{\{\f,\f\}} = \mathbf{c}^{\{\f\}} = 
\begin{bmatrix}
 \left( c^\mathsc{o}_{2} \right)\,c^\mathsc{i}  \\
\vdots \\
c_{i-1}^\mathsc{o}  \, \one +  \left( c^\mathsc{o}_{i}-c^\mathsc{o}_{i-1} \right)\,c^\mathsc{i} \\
\vdots \\
c_{s^\mathsc{o}}^\mathsc{o}  \, \one +  \left( b^\mathsc{o}\,^T c^\mathsc{o} -c^\mathsc{o}_{s^\mathsc{o}}  \right)\,c^\mathsc{i}
\end{bmatrix} \in \Re^{s^\mathsc{o}s^\mathsc{i}}. 
\]
\item
Assuming that both the fast and the slow methods have order at least two,
the simplifying assumptions imply that the overall scheme is second order.
\item
Assuming that both the fast and the slow methods have order at least three,
the third order coupling conditions reduce to the single condition
\begin{eqnarray}
\label{rfsmr.order3.cond}
\frac{1}{3} &=& \sum_{i=2}^{s^\mathsc{o}} \left(c_i^\mathsc{o}-c_{i-1}^\mathsc{o}\right)\, \left(   \mathbf{e}_{i}+\mathbf{e}_{i-1} \right)^T \, A^\mathsc{o} c^\mathsc{o} + \left(1-c_{s^\mathsc{o}}^\mathsc{o}\right) \,  \left( \frac{1}{2} + \mathbf{e}_{s^\mathsc{o}}^T\, A^\mathsc{o} c^\mathsc{o}   \right).
\end{eqnarray}
\end{romannum}
\end{theorem} 
\begin{proof}
 Comparing with~\eqref{eqn:GARK}, the results above show that:
\[
a_{i,j}^{\{\s,\s\}} = a_{i,j}^\mathsc{o}\,,\quad j=1,\dots,i-1\,,
\]
\[
a_{i,(\ell,m)}^{\{\s,\f\}} = b_{m}^\mathsc{i}\, \left(  c^\mathsc{o}_{\ell} - c^\mathsc{o}_{\ell-1} \right)\,,   \quad \ell=2,\dots,i\,, \quad m=1,\dots,s^\mathsc{i}\,,
\]
\[
a_{(i,k),j}^{\{\f,\s\}} = a^\mathsc{o}_{i-1,j}  + c_k^\mathsc{i}\,  \left( a^\mathsc{o}_{i,j} - a^\mathsc{o}_{i-1,j} \right)\,, \quad j=1,\dots,i-1\,, \quad k=1,\dots,s^\mathsc{i}\,,
\]
and
\[
a_{(i,k),(\ell,m)}^{\{\f,\f\}} = 
\left\{ \begin{array}{ll}
b_{m}^\mathsc{i}\,\left(  c^\mathsc{o}_{\ell} - c^\mathsc{o}_{\ell-1} \right)\,, &  \ell=2,\dots,i-1\,, \\
a_{k,m}^\mathsc{i}\,\left(  c^\mathsc{o}_{i} - c^\mathsc{o}_{i-1} \right)\,, & \ell=i\,,
\end{array}\right. \,, \quad m,k=1,\dots,s^\mathsc{i}\,.
\]
The double subscript indices correspond to the double indices 
of the fast stage vectors. 

We see that the method satisfies the simplifying ``internal consistency'' conditions for $i=1,\ldots,s^\mathsc{o}$ and $k=1,\ldots,s^\mathsc{i}$, since
\begin{eqnarray*}
c_{i}^{\{\s,\s\}} & = & c_i^\mathsc{o}\,, \\
c_{i}^{\{\s,\f\}} & = & \sum_{m=1}^{s^\mathsc{i}} b_{m}^\mathsc{i}\, \sum_{\ell=2}^{i} \left(  c^\mathsc{o}_{\ell} - c^\mathsc{o}_{\ell-1} \right) = c_i^\mathsc{o}\,, \\
%
c_{(i,k)}^{\{\f,\s\}} & = & \sum_{ j=1}^{i-1} a^\mathsc{o}_{i-1,j}  + c_k^\mathsc{i}\,  \sum_{ j=1}^{i-1}\left( a^\mathsc{o}_{i,j} - a^\mathsc{o}_{i-1,j} \right) = c^\mathsc{o}_{i-1} + c_k^\mathsc{i}\, \left( c^\mathsc{o}_{i} - c^\mathsc{o}_{i-1} \right)\,,\\
%
c_{(i,k)}^{\{\f,\f\}} &=& \sum_{m=1}^{s^\mathsc{i}} b_{m}^\mathsc{i}\, \sum_{\ell=2}^{i-1} \left(  c^\mathsc{o}_{\ell} - c^\mathsc{o}_{\ell-1} \right) + \sum_{m=1}^{s^\mathsc{i}} a_{k,m}^\mathsc{i}\,\left(  c^\mathsc{o}_{i} - c^\mathsc{o}_{i-1} \right) = 
c^\mathsc{o}_{i-1} +  c_{k}^\mathsc{i}\,\left(  c^\mathsc{o}_{i} - c^\mathsc{o}_{i-1} \right)\,.
\end{eqnarray*}
One internal step is taken to solve the last ODE \eqref{eqn:RFSMR-internal-ode-solution}:
\begin{subequations}
\label{eqn:RFSMR-one-step-solution}
\begin{eqnarray}
Y_{s^\mathsc{o}+1,k}^{\{\f\}} & = & Y_{s^\mathsc{o}}^{\{\s\}} + H\, c_k^\mathsc{i}\, \sum_{j=1}^{s^\mathsc{o}} \left( b^\mathsc{o}_{j} - a^\mathsc{o}_{s^\mathsc{o},j} \right)\, f^{\{\s\}} \left( Y_j^{\{\s\}} \right) \\
\nonumber
&& + H\, \left(  1 - c^\mathsc{o}_{s^\mathsc{o}} \right)\, \sum_{j=1}^{s^\mathsc{i}} a^\mathsc{i}_{k,j} \, f^{\{\f\}} \left( Y_{s^\mathsc{o}+1,j}^{\{\f\}} \right)\,, \quad k=1,\dots,s^\mathsc{i}  \\ 
y_{n+1} &=& Y_{s^\mathsc{o}}^{\{\s\}} +  H\, \sum_{j=1}^{s^\mathsc{o}} \left( b^\mathsc{o}_{j} - a^\mathsc{o}_{s^\mathsc{o},j} \right)\, f^{\{\s\}} \left( Y_j^{\{\s\}} \right) \\
\nonumber
&& + H\, \left(   1 - c^\mathsc{o}_{s^\mathsc{o}} \right)\, \sum_{j=1}^{s^\mathsc{i}} b_{j}^\mathsc{i}\, f^{\{\f\}} \left( Y_{s^\mathsc{o}+1,j}^{\{\f\}} \right) 
\end{eqnarray}
\end{subequations}
Using
\begin{eqnarray*}
 Y_{s^\mathsc{o}}^{\{\s\}} &=& y_n +  H\, \sum_{j=1}^{s^\mathsc{o}-1}  a^\mathsc{o}_{s^\mathsc{o},j} \, f^{\{\s\}} \left( Y_j^{\{\s\}} \right) 
+ H\,  \sum_{\ell=2}^{s^\mathsc{o}} \, \sum_{j=1}^{s^\mathsc{i}} \, \left(  c^\mathsc{o}_{\ell} - c^\mathsc{o}_{\ell-1} \right)\,  b_{j}^\mathsc{i}\, f^{\{\f\}} \left( Y_{\ell,j}^{\{\f\}}  \right)  \,
\end{eqnarray*}
equation  \eqref{eqn:RFSMR-one-step-solution} becomes
\begin{subequations}
\label{eqn:RFSMR-one-step-solution-standard}
\begin{eqnarray}
Y_{s^\mathsc{o}+1,k}^{\{\f\}} & = & y_n +  H\,  \sum_{j=1}^{s^\mathsc{o}} \left( c_k^\mathsc{i}\,b^\mathsc{o}_{j} + (1- c_k^\mathsc{i})\,a^\mathsc{o}_{s^\mathsc{o},j} \right)\, f^{\{\s\}} \left( Y_j^{\{\s\}} \right) \\
\nonumber
&& + H\,  \sum_{\ell=2}^{s^\mathsc{o}} \, \sum_{j=1}^{s^\mathsc{i}} \, \left(  c^\mathsc{o}_{\ell} - c^\mathsc{o}_{\ell-1} \right)\,  b_{j}^\mathsc{i}\, f^{\{\f\}} \left( Y_{\ell,j}^{\{\f\}}  \right) \\
\nonumber
&& + H\, \left(  1 - c^\mathsc{o}_{s^\mathsc{o}} \right)\, \sum_{j=1}^{s^\mathsc{i}} a^\mathsc{i}_{k,j} \, f^{\{\f\}} \left( Y_{s^\mathsc{o}+1,j}^{\{\f\}} \right)\,, \quad k=1,\dots,s^\mathsc{i}  \\ 
y_{n+1} &=& y_n +    H\, \sum_{j=1}^{s^\mathsc{o}}  b^\mathsc{o}_{j} \, f^{\{\s\}} \left( Y_j^{\{\s\}} \right) \\
\nonumber
&& + H\,  \sum_{\ell=2}^{s^\mathsc{o}} \, \sum_{j=1}^{s^\mathsc{i}} \, \left(  c^\mathsc{o}_{\ell} - c^\mathsc{o}_{\ell-1} \right)\,  b_{j}^\mathsc{i}\, f^{\{\f\}} \left( Y_{\ell,j}^{\{\f\}}  \right) \\
\nonumber
&& + H\, \left(   1 - c^\mathsc{o}_{s^\mathsc{o}} \right)\, \sum_{j=1}^{s^\mathsc{i}} b_{j}^\mathsc{i}\, f^{\{\f\}} \left( Y_{s^\mathsc{o}+1,j}^{\{\f\}} \right) 
\end{eqnarray}
\end{subequations}

The coefficients are
\[
b_{i}^{\{\s\}} = b_i^\mathsc{o}\,, \quad i=1,\dots,s^\mathsc{o}\,,
\]
\[
b_{(\ell,m)}^{\{\f\}} = 
\left\{ \begin{array}{ll}
b_{m}^\mathsc{i}\,\left(  c^\mathsc{o}_{\ell} - c^\mathsc{o}_{\ell-1} \right)\,, &  \ell=2,\dots,s^\mathsc{o}\,, \\
b_{m}^\mathsc{i}\,\left(  1 - c^\mathsc{o}_{s^\mathsc{o}} \right)\,, & \ell=s^\mathsc{o}+1\,,
\end{array}\right. \,, \quad m=1,\dots,s^\mathsc{i}\,.
\]
\[
a_{(s^\mathsc{o}+1,k),j}^{\{\f,\s\}} = a^\mathsc{o}_{s^\mathsc{o},j}  + c_k^\mathsc{i}\,  \left( b^\mathsc{o}_{j} - a^\mathsc{o}_{s^\mathsc{o},j} \right)\,, \quad j=1,\dots,s^\mathsc{o}\,, \quad k=1,\dots,s^\mathsc{i}\,,
\]
and
\[
a_{(s^\mathsc{o}+1,k),(\ell,m)}^{\{\f,\f\}} = 
\left\{ \begin{array}{ll}
b_{m}^\mathsc{i}\,\left(  c^\mathsc{o}_{\ell} - c^\mathsc{o}_{\ell-1} \right)\,, &  \ell=2,\dots,s^\mathsc{o}\,, \\
a_{k,m}^\mathsc{i}\,\left(  1 - c^\mathsc{o}_{s^\mathsc{o}} \right)\,, & \ell=s^\mathsc{o}\,,
\end{array}\right. \,, \quad m,k=1,\dots,s^\mathsc{i}\,.
\]

We have that the first order conditions as well as the last simplifying conditions hold
\begin{eqnarray*}
\sum_{\ell=2}^{s^\mathsc{o}} \sum_{m=1}^{s^\mathsc{i}} b_{(\ell,m)}^{\{\f\}} &=& \sum_{\ell=2}^{s^\mathsc{o}} \left( c^\mathsc{o}_{\ell} - c^\mathsc{o}_{\ell-1}\right) + 1 - c^\mathsc{o}_{s^\mathsc{o}} =1\,, \\
c_{(s^\mathsc{o}+1,k)}^{\{\f,\s\}} &=&  \sum_{j=1}^{s^\mathsc{o}} a^\mathsc{o}_{s^\mathsc{o},j}  + c_k^\mathsc{i}\,  \sum_{j=1}^{s^\mathsc{o}} \left( b^\mathsc{o}_{j} - a^\mathsc{o}_{s^\mathsc{o},j} \right) = c^\mathsc{o}_{s^\mathsc{o}} + c_k^\mathsc{i}\, \left(1 - c^\mathsc{o}_{s^\mathsc{o}} \right)\, \\
c_{(s^\mathsc{o}+1,k)}^{\{\f,\f\}} &=& \sum_{\ell=2}^{s^\mathsc{o}} \left(  c^\mathsc{o}_{\ell} - c^\mathsc{o}_{\ell-1} \right) + c_{k}^\mathsc{i}\,\left(  1 - c^\mathsc{o}_{s^\mathsc{o}} \right) = c^\mathsc{o}_{s^\mathsc{o}} + c_k^\mathsc{i}\, \left(1 - c^\mathsc{o}_{s^\mathsc{o}} \right)\,.
\end{eqnarray*}
Due to the GARK properties \cite{SaGu13a} the simplifying conditions imply that the overall scheme is second order,
assuming that both the fast and the slow methods have order at least two.
This proves parts (ii) and (iii).

Collecting all the coefficients together we have that the slow method is the outer method, and therefore has order $p$:
\begin{eqnarray*}
a_{i,j}^{\{\s,\s\}} &=& a_{i,j}^\mathsc{o}\,, \quad i=1,\dots,s^\mathsc{o}\,,\quad j=1,\dots,i-1\,, \\
b_{i}^{\{\s\}} &=& b_i^\mathsc{o}\,, \quad i=1,\dots,s^\mathsc{o}\,.
\end{eqnarray*}

The inner scheme takes $s^\mathsc{o}$ consecutive steps of the fast method ($i = 2,\dots,s^\mathsc{o}+1$) each with a step size $ c^\mathsc{o}_{i} - c^\mathsc{o}_{i-1}$,
and therefore also has order $p$:
\begin{eqnarray*}
a_{(i,k),(\ell,m)}^{\{\f,\f\}} &=& 
\left\{ \begin{array}{ll}
b_{m}^\mathsc{i}\,\left(  c^\mathsc{o}_{\ell} - c^\mathsc{o}_{\ell-1} \right)\,, &   i=2,\dots, s^\mathsc{o}+1\,, ~~ \ell=2,\dots,i-1\,, \\
a_{k,m}^\mathsc{i}\,\left(  c^\mathsc{o}_{i} - c^\mathsc{o}_{i-1} \right)\,, & i=2,\dots, s^\mathsc{o}\,, ~~ \ell=i\,, \\
a_{k,m}^\mathsc{i}\,\left(  1 - c^\mathsc{o}_{s^\mathsc{o}} \right)\,, & i=s^\mathsc{o}+1\,, ~~ \ell=s^\mathsc{o}\,,
\end{array}\right.  \\
&& \quad m,k=1,\dots,s^\mathsc{i}\\
b_{(\ell,m)}^{\{\f\}} &=& 
\left\{ \begin{array}{ll}
b_{m}^\mathsc{i}\,\left(  c^\mathsc{o}_{\ell} - c^\mathsc{o}_{\ell-1} \right)\,, &  \ell=2,\dots,s^\mathsc{o}\,, \\
b_{m}^\mathsc{i}\,\left(  1 - c^\mathsc{o}_{s^\mathsc{o}} \right)\,, & \ell=s^\mathsc{o}+1\,,
\end{array}\right. \,, \quad m=1,\dots,s^\mathsc{i}\,.
\end{eqnarray*}

The coupling coefficients are
\begin{eqnarray*}
a_{i,(\ell,m)}^{\{\s,\f\}} &=& b_{m}^\mathsc{i}\, \left(  c^\mathsc{o}_{\ell} - c^\mathsc{o}_{\ell-1} \right)\,,   \\
&& \quad i=2,\dots,s^\mathsc{o}\,, ~~ \ell=2,\dots,i\,, ~~ m=1,\dots,s^\mathsc{i}\,,
\end{eqnarray*}
and
\begin{eqnarray*}
a_{(i,k),j}^{\{\f,\s\}} &=& 
\left\{ \begin{array}{ll}
a^\mathsc{o}_{i-1,j}  + c_k^\mathsc{i}\,  \left( a^\mathsc{o}_{i,j} - a^\mathsc{o}_{i-1,j} \right)\,, \quad i=2,\dots,s^\mathsc{o}\,, ~~ j=1,\dots,i-1\,,  \\
a^\mathsc{o}_{s^\mathsc{o},j}  + c_k^\mathsc{i}\,  \left( b^\mathsc{o}_{j} - a^\mathsc{o}_{s^\mathsc{o},j} \right)\,, \quad i=s^\mathsc{o}+1\,, ~~  j=1,\dots,s^\mathsc{o}\,
\end{array}\right. \\
&& \quad k=1,\dots,s^\mathsc{i}\,,
\end{eqnarray*}
which proves part (i).

To verify part (iv), we see that
the third order coupling conditions reduce to
\begin{eqnarray*}
\left(\mathbf{b}^{\{\s\}}\right)^T \cdot \mathbf{A}^{\{\s,\f\}}  \cdot \mathbf{c}^{\{\f\}}   & = & \frac{1}{6}\,, \\ 
\mathbf{b}^{\{\f\}}\,^T \cdot \mathbf{A}^{\{\f,\s\}}  \cdot \mathbf{c}^{\{\s\}}   & = & \frac{1}{6}\,, 
\end{eqnarray*}
assuming that both the fast and the slow methods have order at least three.
The first condition is satisfied automatically since
\begin{eqnarray*}
\mathbf{A}^{\{\s,\f\}}  \cdot \mathbf{c}^{\{\f\}}   & = & \sum_{i=2}^{s^\mathsc{o}}  \left( \left(c_i^\mathsc{o}-c_{i-1}^\mathsc{o}\right)\, \mathbf{g}_i b^\mathsc{i}\,^T \right) \cdot \left( c_{i-1}^\mathsc{o}  \, \one +  \left( c^\mathsc{o}_{i}-c^\mathsc{o}_{i-1} \right)\,c^\mathsc{i} \right) \\
& = & \sum_{i=2}^{s^\mathsc{o}}  c_{i-1}^\mathsc{o}\, \left(c_i^\mathsc{o}-c_{i-1}^\mathsc{o}\right)\, \mathbf{g}_i    +  \frac{1}{2} \left( c^\mathsc{o}_{i}-c^\mathsc{o}_{i-1} \right)^2\, \mathbf{g}_i  \\
& = & \sum_{i=2}^{s^\mathsc{o}}  \frac{1}{2} \left( (c^\mathsc{o}_{i})^2-(c^\mathsc{o}_{i-1})^2 \right)\, \mathbf{g}_i  \\
&=& \frac{1}{2} \, (c^\mathsc{o})^2 \\
\left(\mathbf{b}^{\{\s\}}\right)^T \cdot \mathbf{A}^{\{\s,\f\}}  \cdot \mathbf{c}^{\{\f\}}   & = &  \frac{1}{2} (b^\mathsc{o})^T (c^\mathsc{o})^2 = \frac{1}{6}\,,
\end{eqnarray*}
where the square is taken component-wise.

For the second condition
%
%
%
%
\begin{eqnarray*}
\mathbf{b}^{\{\f\}}\,^T \cdot \mathbf{A}^{\{\f,\s\}}  \cdot \mathbf{c}^{\{\s\}}  
& = & 
\begin{bmatrix}
c_2^\mathsc{o}\,  b^\mathsc{i}\,^T & \left(c_3^\mathsc{o}-c_{2}^\mathsc{o}\right)\, b^\mathsc{i}\,^T & \ldots & \left(1-c_{s^\mathsc{o}}^\mathsc{o}\right) \, b^\mathsc{i}\,^T \end{bmatrix} \cdot \\
& & \cdot 
\begin{bmatrix}
\one\, \mathbf{e}_{1}^T\, A^\mathsc{o} c^\mathsc{o} + c^\mathsc{i}\,  \left( \mathbf{e}_{2}^T-\mathbf{e}_{1}^T \right)\, A^\mathsc{o} c^\mathsc{o} \\
\vdots \\
\one\, \mathbf{e}_{i-1}^T\, A^\mathsc{o} c^\mathsc{o}  + c^\mathsc{i}\,  \left( \mathbf{e}_{i}^T-\mathbf{e}_{i-1}^T \right)\, A^\mathsc{o} c^\mathsc{o} \\
\vdots \\
\one\, \mathbf{e}_{s^\mathsc{o}}^T\, A^\mathsc{o} c^\mathsc{o}  + c^\mathsc{i}\,  \left(\frac{1}{2} -\mathbf{e}_{s^\mathsc{o}}^T\, A^\mathsc{o} c^\mathsc{o}  \right)
\end{bmatrix} \\
& = &
\sum_{i=2}^{s^\mathsc{o}} \left(c_i^\mathsc{o}-c_{i-1}^\mathsc{o}\right)\, b^\mathsc{i}\,^T \left( \one\, \mathbf{e}_{i-1}^T  + c^\mathsc{i}\,  \left( \mathbf{e}_{i}^T-\mathbf{e}_{i-1}^T \right) \right) \, A^\mathsc{o} c^\mathsc{o} \\
&& + \left(1-c_{s^\mathsc{o}}^\mathsc{o}\right) \, b^\mathsc{i}\,^T \left( \one\, \mathbf{e}_{s^\mathsc{o}}^T\, A^\mathsc{o} c^\mathsc{o}  + c^\mathsc{i}\,  \left(\frac{1}{2} -\mathbf{e}_{s^\mathsc{o}}^T\, A^\mathsc{o} c^\mathsc{o}  \right) \right) \\
& = &
\frac{1}{2}\, \sum_{i=2}^{s^\mathsc{o}} \left(c_i^\mathsc{o}-c_{i-1}^\mathsc{o}\right)\, \left(   \mathbf{e}_{i}^T+\mathbf{e}_{i-1}^T \right) \, A^\mathsc{o} c^\mathsc{o} \\
&& + \left(1-c_{s^\mathsc{o}}^\mathsc{o}\right) \,  \left( \frac{1}{4} + \frac{1}{2}\, \mathbf{e}_{s^\mathsc{o}}^T\, A^\mathsc{o} c^\mathsc{o}   \right),
\end{eqnarray*}
which gives ~\eqref{rfsmr.order3.cond}. \qquad
\end{proof}
\begin{remark}
The condition~\eqref{rfsmr.order3.cond} corresponds to the additional order three condition derived in the original MIS paper~\cite{Knoth_1998_MRimex}.
\end{remark}
\section{Multirate (traditional) additive Runge-Kutta methods}\label{sec:mark}

A special case of multirate GARK schemes~\eqref{MGARK}
are the multirate additive Runge-Kutta schemes. They are obtained in the GARK framework by setting
\begin{eqnarray*}
&&A^{\{\s\}}:=A^{\{\s,\s\}}=A^{\{\f,\s,1\}}, \\
&&A^{\{\f\}}:=A^{\{\f,\f\}}=A^{\{\s,\f,1\}}, \\
&&A^{\{\s,\f,\lambda\}}:=0  \quad  \textnormal{and} \quad
A^{\{\f,\s,\lambda\}}:=A^{\{\s,\lambda\}} \quad \textnormal{for}\quad \lambda=2,\ldots,M.
\end{eqnarray*}
The scheme proceeds as follows
\begin{subequations}
\label{multaddRKscheme}
\begin{eqnarray}
Y_i^{\lambda} & = & y_n + h \sum_{l=1}^{\lambda-1} \sum_{j=1}^s b_{j}^{\{\f\}} f^{\{\f\}}(Y_j^{l}) +  H \, \sum_{j=1}^s a_{i,j}^{\{\s,\lambda\}} f^{\{\s\}}(Y_j^{1}) \\ \nonumber 
& &  + h \, \sum_{j=1}^s a_{i,j}^{\{\f\}} f^{\{\f\}}(Y_j^{\lambda}), \quad \lambda=1,\ldots,M,\\
y_{n+1} & = & 
y_n + h \, \sum_{\lambda=1}^M \sum_{j=1}^s b_{i}^{\{\f\}} f^{\{\f\}}(Y_i^{\lambda}) + H \, \sum_{j=1}^s b_{i}^{\{\s\}} f^{\{\s\}}(Y_i^{1}),
\end{eqnarray}
\end{subequations}
and has the following extended Butcher tableau
\renewcommand{\arraystretch}{1.5}
\[
\begin{array}{cccc|cccc}
\frac{1}{M} A^{\{\f\}} & 0 & \cdots & 0 
& A^{\{\s\}} & 0 & \cdots & 0 \\
\frac{1}{M} \one b^{\{\f\}}\,^T & \frac{1}{M} A^{\{\f\}} & \cdots & 0 
& A^{\{\s,2\}} & 0 & \cdots & 0 \\
\vdots & & \ddots & &
\vdots & & & \vdots \\
\frac{1}{M} \one b^{\{\f\}}\,^T & \frac{1}{M} \one b^{\{\f\}}\,^T & \cdots & 
\frac{1}{M} A^{\{\f\}}
& A^{\{\s,M\}} & 0 & \cdots & 0 \\ \hline
\frac{1}{M} b^{\{\f\}}\,^T & \frac{1}{M} b^{\{\f\}}\,^T & \cdots & \frac{1}{M} b^{\{\f\}}\,^T
& b^{\{\s\}}\,^T & 0 & \cdots & 0
\end{array}
\]
\renewcommand{\arraystretch}{1.00}
\subsection{Additive partitioning}
The coupling condition \eqref{eqn:aas-for-stability} for nonlinear stability yields 
\begin{eqnarray*}
 A^{\{\s\}} & = & \one b^{\{\s\}}\,^T - B^{\{\f\}}\,^{-1} (A^{\{\f\}})^T B^{\{\s\}}, \\
A^{\{\s,\lambda\}} & = & \one b^{\{\s\}}\,^T\,, \qquad \lambda=2,\ldots,M,
\end{eqnarray*}
and only $b^{\{\s\}},b^{\{\f\}}$ and $A^{\{\f\}}$ remain as free parameters.
As a consequence, the algebraic stability of the basic method 
$(b^{\{\s\}},A^{\{\s\}})$ is equivalent to the algebraic stability of $(b^{\{\s\}},D)$, where $D=B^{\{\f\}}\,^{-1} (A^{\{\f\}})^T B^{\{\s\}} $. 
\begin{example}[A nonlinearly stable additive Runge-Kutta scheme of order two]
Besides the algebraic stability of $(b^{\{\s\}},A^{\{\s\}})$ and $(b^{\{\s\}},D)$,
the following conditions have to be fulfilled for a method of order two:
\renewcommand{\arraystretch}{1.5}
\begin{eqnarray*}
b^{\{\f\}}\,^T \one & = & 1, \\
b^{\{\f\}}\,^T A^{\{\f\}} \one & = & \frac{1}{2}, \\
b^{\{\s\}}\,^T \one & = & 1, \\
b^{\{\s\}}\,^T D \one & = & \frac{1}{2}, \\
b^{\{\s\}}\,^T A^{\{\f\}} \one & = &\frac{M}{2}.
\end{eqnarray*}
\renewcommand{\arraystretch}{1.0}
A simple choice of parameters is
\renewcommand{\arraystretch}{1.5}
\begin{eqnarray*} 
b^{\{\f\}} & = & \begin{bmatrix} \frac{1}{2} \\ \frac{1}{2} \end{bmatrix}, \quad
b^{\{\s\}}  =  \begin{bmatrix} \frac{3}{4M+2} \\ \frac{4M-1}{4M+2} \end{bmatrix}, \quad
A^{\{\f\}}  =  \begin{bmatrix} \frac{1}{4} & - \frac{M}{2} \\
\frac{M+1}{2} & \frac{1}{4} \end{bmatrix},
\end{eqnarray*}
\renewcommand{\arraystretch}{1.0}
and
\renewcommand{\arraystretch}{1.5}
\begin{eqnarray*}
A^{\{\s\}} & = & \frac{1}{4M+2}  \begin{bmatrix} \frac{3}{2} & - M(4M-1) \\
3(M+1) & \frac{4M-1}{2} \end{bmatrix}
\end{eqnarray*}
\renewcommand{\arraystretch}{1.0}
In this case both base methods are not only algebraically stable but also symplectic.
\end{example}

\subsection{Componentwise partitioning}
For component partitioning there are no additional nonlinear stability conditions. 
Following again the lines of~\cite{Kvaerno_1999_MR-RK}, we set 
\begin{subequations}
\label{component-special-conditions}
\begin{equation}
A^{\{\s,\lambda+1\}}=A^{\{\s\}} + F(\lambda)\quad \textnormal{with} \quad F(\lambda) \one = \lambda
 \one, \quad \lambda=0,\ldots,M-1.
\end{equation}
We also consider the simplifying assumption 
\begin{equation}
c:= A^{\{\f\}} \one =  A^{\{\s\}} \one\,.
\end{equation}
\end{subequations}
The corresponding order conditions are given in Table~\ref{table_multaddRK_2a}.

\begin{table}
\[
\renewcommand{\arraystretch}{1.5}
\begin{array}{ccrclrcl}
\mbox{No.} & \mbox{Order} & \multicolumn{3}{c}{\mbox{order cond. (slow part)}}& \multicolumn{3}{c}{\mbox{order cond. (fast part)}} \\ \hline
i) & 1 & b^{\{\s\}}\,^T \one & = & 1 & b^{\{\f\}}\,^T \one & = & 1\\ \hline
ii) & 2 & b^{\{\s\}}\,^T c & = & \frac{1}{2} & b^{\{\f\}}\,^T c & = & \frac{1}{2}\\ 
iv) & 3 & b^{\{\s\}}\,^T \diag(c) c & =&  \frac{1}{3} & b^{\{\f\}}\,^T \diag(c) c & =&  \frac{1}{3}\\
vii) &  & b^{\{\s\}}\,^T A^{\{\s\}} c & = & \frac{1}{6} & b^{\{\f\}}\,^T A^{\{\f\}} c & = & \frac{1}{6}\\
ix) &  & b^{\{\s\}}\,^T A^{\{\f\}} c & = & \frac{M}{6} & b^{\{\f\}}\,^T \left(A^{\{\s\}} + \frac{1}{M} \sum_{\lambda=0}^{M-1} F(\lambda) \right) c & = & \frac{M}{6} \\
\end{array}
\renewcommand{\arraystretch}{1.0}
\]
\caption{\label{table_multaddRK_2a} Order conditions for multirate additive Runge Kutta scheme~\eqref{multaddRKscheme} when \eqref{component-special-conditions} hold.}
\end{table}


If, in addition to \eqref{component-special-conditions}, 
a second simplifying condition is given by the following relation
\begin{eqnarray}
\label{component-special-conditionb}
b^{\{\f\}}\,^T \sum_{\lambda=0}^{M-1} F(\lambda) \, c & = & 0. 
\end{eqnarray}
When \eqref{component-special-conditions} and \eqref{component-special-conditionb} hold
any pair $(b^{\{\f\}}, A^{\{\f\}})$ and $(b^{\{\s\}}, A^{\{\s\}})$ of algebraically stable order three schemes lead to an order three multirate scheme, provided that the  compatibility conditions are true:
\begin{eqnarray*}
b^{\{\s\}}\,^T A^{\{\f\}} \,c & = & \frac{M}{6}, \quad
b^{\{\f\}}\,^T A^{\{\s\}} \,c  =  \frac{M}{6}.
\end{eqnarray*}
These compatibility conditions are fulfilled by a scheme with $s$ stages if
\[
\eta_1(\lambda) = \frac{c_{s} \sum_{j=2}^{s-1} \eta_j - \sum_{j=2}^{s-1} c_{j} \eta_j - c_{s} \lambda}{c_{1} - c_{s}}, \quad
\eta_s(\lambda)=\lambda - \sum_{j=1}^{s-1} \eta_j,
\]
with free parameters $\eta_2(\lambda),\ldots,\eta_{s-1}(\lambda)$, 
provided that $c_1 \neq c_s$.

It is easy to see that the scheme must have at least four stages, as $s=3$ would yield $b^{\{\f\}}=b^{\{\s\}}$ and consequently $M=1$.

\begin{example}[An algebraically stable additive Runge-Kutta scheme of order three]
To construct an algebraically stable scheme of order $p=3$, we first choose a pair of algebraically stable schemes $(A,b^{\{\f\}})$ and $(A,b^{\{\s\}})$ with $c:=A \one$ and then define 
$A^{\{\f\}}:=A+(M-1) \widetilde{A}^{\{\f\}}$ and $A^{\{\s\}}:=A+(M-1) \widetilde{A}^{\{\s\}}$.
It is straightforward to show that this yields a stable multirate additive Runge-Kutta scheme, if the following conditions hold:
\begin{subequations}
\begin{eqnarray}
\widetilde{A}^{\{\f\}} \one & = & \widetilde{A}^{\{\s\}} \one = 0, \\
b^{\{\f\}} \widetilde{A}^{\{\f\}} c & = & b^{\{\s\}} \widetilde{A}^{\{\s\}} c = 0, \\
b^{\{\s\}} \widetilde{A}^{\{\f\}} c & = & b^{\{\f\}} \widetilde{A}^{\{\s\}} c = \frac{1}{6}, \\
\widetilde{A}^{\{\f\}}\,^T B^{\{\f\}} + B^{\{\f\}} A^{\{\f\}} & = &
\widetilde{A}^{\{\s\}}\,^T B^{\{\s\}} + B^{\{\s\}} A^{\{\s\}} = 0.
\end{eqnarray} 
\end{subequations}
Such a pair can be constructed by extending (doubling) any algebraically stable scheme. For the RADAU-IA method, for example, the extension to the pair $(A,b^{\{\f\}})$ and $(A,b^{\{\s\}})$ is given by:
\renewcommand{\arraystretch}{1.25}
\[
\begin{array}{c|rrrrr}
0 & \frac{1}{4} & -\frac{1}{4} & 0 & 0 \\
\frac{2}{3} & \frac{1}{4} & \frac{5}{12} & 0 & 0 \\
0 & 0 & 0 & \frac{1}{4} & -\frac{1}{4} \\
\frac{2}{3} & 0 & 0 &  \frac{1}{4} & \frac{5}{12} \\ \hline 
b^{\{\f\}} & \frac{1}{4} & \frac{3}{4} & 0 & 0 \\
b^{\{\s\}} & 0 & 0 & \frac{1}{4} & \frac{3}{4}
\end{array}\,.
\]
A possible choice for $\widetilde{A}^{\{\f\}}$ and $\widetilde{A}^{\{\s\}}$ fulfilling all conditions above is
\[
\widetilde{A}^{\{\f\}} = \begin{bmatrix}
0 & 0  & ~~0 & 0 \\
0 & 0  &~~ 0 & 0 \\
0 & 0 & ~~0 & 0 \\
0 & 0 & -\frac{1}{3} & \frac{1}{3}
\end{bmatrix}, \qquad
\widetilde{A}^{\{\s\}} = \begin{bmatrix}
~~0 & 0 & 0 & 0 \\
-\frac{1}{3} & \frac{1}{3} & 0 & 0 \\
~~0 & 0 & 0 & 0 \\
~~0 & 0 & 0 & 0
\end{bmatrix}.
\]
Finally, we set 
\[
\eta_1(\lambda) = \frac{c_{4} \sum_{j=2}^{3} \eta_j - \sum_{j=2}^{3} c_{j} \eta_j - c_{4} \lambda}{c_{1} - c_{4}}, \quad
\eta_4(\lambda)=\lambda - \sum_{j=1}^{3} \eta_i,
\]
with $\eta_2$ and $\eta_3$ arbitrary. For $\eta_2:=\eta_3:=0$ we get
$\eta_1=\frac{c_4}{c_4-c_1} \lambda$ and
$\eta_4=-\frac{c_1}{c_4-c_1} \lambda$.
For the extended RADAU-IA scheme we have $\eta_1=\lambda, \eta_2=\eta_3=\eta_4=0$.
\end{example}

\section{Monotonicity properties}\label{sec:monotonicity}

Consider the method \eqref{MGARK} in the general form, represented by the Butcher tableau \eqref{eqn:mrRK-butcher}
and let
%
%
\[
\renewcommand{\arraystretch}{1.5}
\widetilde{\mathbf{A}} = \begin{bmatrix} 
\frac{1}{M} A^{\{\f,\f\}}                       & \cdots & 0 &  A^{\{\f,\s,1\}}  & 0 \\
\vdots                               & \ddots &   & \vdots & 0  \\
\frac{1}{M} \mathbf{1} b^{\{\f\}}\,^T   & \ldots & \frac{1}{M} A^{\{\f,\f\}} &  A^{\{\f,\s,M\}}& 0  \\
\frac{1}{M} A^{\{\s,\f,1\}}  & \cdots & \frac{1}{M} A^{\{\s,\f,M\}} & A^{\{\s,\s\}} & 0   \\   
\frac{1}{M} b^{\{\f\}}\,^T  & \ldots & \frac{1}{M} b^{\{\f\}}\,^T & b^{\{\s\}}\,^T  & 0 
 \end{bmatrix}\,.
\renewcommand{\arraystretch}{1.0}
\]

We are concerned with partitioned systems \eqref{eqn:additive-ode} where there exist $\rho > 0$ such that 
for any $y$
\begin{equation} 
\label{eqn:monotone-ode}
\left\Vert y + \rho\, f^{\{\s\}} (y) \right\Vert \le \left\Vert y  \right\Vert\,, \quad
\left\Vert y + \frac{\rho}{M}\, f^{\{\f\}} (y) \right\Vert \le \left\Vert y  \right\Vert \,.
\end{equation}
This implies that condition \eqref{eqn:monotone-ode} holds for any $0 \le \tau \le \rho$, i.e., the solutions of forward Euler steps
with the slow and fast subsystems, respectively, are monotone
under this step size restriction. The condition \eqref{eqn:monotone-ode} also implies that
the system \eqref{eqn:additive-ode} has a solution of non increasing norm.
To see this consider $\alpha,\beta > 0$ with $\alpha+\beta=1$ and write an Euler step with the full system as a convex combination
\begin{eqnarray*}
\left\Vert y + \theta\, \left( f^{\{\s\}} (y) + f^{\{\f\}} (y)\right)  \right\Vert &=& \left\Vert \alpha\, \left(y + \frac{\theta}{\alpha}\, f^{\{\s\}} (y)\right) + 
\beta \left(y + \frac{\theta}{\beta}\, f^{\{\f\}} (y) \right) \right\Vert \\
&\le& \alpha \left\Vert y \right\Vert + \beta \left\Vert y \right\Vert  = \left\Vert y \right\Vert,
\end{eqnarray*}
if $\theta \le \min\{\alpha , \beta/M\}\,  \rho$. For $\beta \to 1$ the aggregate Euler step is monotonic for time steps $0 < \theta < \rho/M$, 
and therefore the solution of \eqref{eqn:additive-ode} has non increasing norm, $(d/dt) \left\Vert y  \right\Vert \le 0$ \cite{Higueras_2006_SSP-ARK}.

We seek multirate schemes which guarantee a monotone numerical solution \linebreak $\left\Vert y_{n+1}  \right\Vert \le \left\Vert y_{n}  \right\Vert$ for \eqref{eqn:monotone-ode} under suitable step size restrictions. Specifically, we seek schemes where the macro step is not subject to the above $\theta < \rho/M$ bound.

The following definition and results follow from the general ones for GARK schemes in ~\cite{SaGu13a}.
\begin{definition}[Absolutely monotonic multirate GARK] 
Let $r>0$ and 
\begin{equation}
\label{eqn:R-tilde}
\widetilde{\mathbf{R}}=\textnormal{diag}\left\{ M\, r\,\mathbf{I}_{M s^{\{\f\}} \times  M s^{\{\f\}}},  r\,\mathbf{I}_{s^{\{\s\}} \times  s^{\{\s\}}} , 1 \right\}\,.
\end{equation}
A multirate GARK scheme \eqref{eqn:GARK} defined by $\widetilde{\mathbf{A}} \ge 0$ is called {\em absolutely monotonic} (a.m.) at $r \in \Re$ if 
\begin{subequations}
\label{eqn:am-conditions}
\begin{eqnarray}
\label{eqn:am-condition-e}
\alpha(r) &=&  \left(\mathbf{I}_{\hat{s} \times \hat{s}} + \widetilde{\mathbf{A}}\widetilde{\mathbf{R}} \right)^{-1} \cdot  \mathbf{1}_{\hat{s} \times 1} \ge 0\,, ~~~\textnormal{and} \\
\label{eqn:am-condition-alpha}
\beta(r) &=&  \left(\mathbf{I}_{\hat{s} \times \hat{s}} + \widetilde{\mathbf{A}}\widetilde{\mathbf{R}} \right)^{-1} \cdot  \widetilde{\mathbf{A}}\widetilde{\mathbf{R}}
=  \mathbf{I}_{\hat{s} \times \hat{s}} - \left(\mathbf{I}_{\hat{s} \times \hat{s}} + \widetilde{\mathbf{A}}\widetilde{\mathbf{R}} \right)^{-1} \ge 0\,,
\end{eqnarray}
\end{subequations}
where  $\hat{s}=M s^{\{\f\}}+ s^{\{\s\}} +1$. Here all the inequalities are taken component-wise.
\end{definition}

Let 
\begin{equation}
\label{eqn:mr-Ahat}
\widehat{\mathbf{A}} =  
\widetilde{\mathbf{A}} \cdot \textnormal{diag}\left\{ M\,\mathbf{I}_{M s^{\{\f\}} \times  M s^{\{\f\}}},  \mathbf{I}_{s^{\{\s\}} \times  s^{\{\s\}}} , 1 \right\}.
\end{equation}
We note that conditions \eqref{eqn:am-conditions} are equivalent to
\begin{subequations}
\label{eqn:am-simple-conditions}
\begin{eqnarray}
\label{eqn:am-simple-condition-e}
\alpha(r) &=&  \left(\mathbf{I}_{\hat{s} \times \hat{s}} + r\,\widehat{\mathbf{A}} \right)^{-1} \cdot  \mathbf{1}_{\hat{s} \times 1} \ge 0\,, ~~~\textnormal{and} \\
\label{eqn:am-simple-condition-alpha}
\beta(r) &=&  \left(\mathbf{I}_{\hat{s} \times \hat{s}} + r\, \widehat{\mathbf{A}} \right)^{-1} \cdot  \widehat{\mathbf{A}} \ge 0\,,
\end{eqnarray}
\end{subequations}
These are precisely the monotonicity relations for a simple Runge Kutta matrix. Consequently, the machinery developed for assessing the monotonicity
of single rate Runge Kutta schemes \cite{Higueras_2004_SSP,Kraaijevanger_1991_contractivity} can be directly applied to 
the multirate GARK case as well.

\begin{definition}[Radius of absolute monotonicity] 
The radius of absolute monotonicity of the multirate GARK scheme \eqref{eqn:GARK}  is the largest number $\mathcal{R}\ge 0$
such that the scheme is absolutely monotonic \eqref{eqn:am-simple-conditions} for any 
$r \in [0,\mathcal{R}]$.
\end{definition}

\begin{theorem}[Monotonicity of solutions] 
Consider the GARK scheme  \eqref{eqn:GARK} 
applied to a partitioned system with the property \eqref{eqn:monotone-ode}. 
For any macro-step size obeying the restriction
\begin{equation}
\label{eqn:step-restriction-monotonicity}
H \le  \mathcal{R}\, \rho 
\end{equation}
the stage values and the solution are monotonic
\begin{subequations}
\label{eqn:monotonicity-conclusions}
\begin{eqnarray}
\left\Vert  Y_i^{\{q\}} \right\Vert &\le& \left\Vert y_n \right\Vert\,, \quad q=1,\dots,N, ~~i=1,\dots,s^{\{q\}}\,, ~~ q \in \{\f,\s\} \\
\ \left\Vert  y_{n+1} \right\Vert &\le& \left\Vert y_n \right\Vert\,.
\end{eqnarray}
\end{subequations}
\end{theorem}
In practice we are interested in the largest upper bound for the time step that ensures monotonicity.

We next consider the particular case of telescopic multirate GARK methods,
 where both the fast and the slow components use the same 
irreducible monotonic base scheme
$(A,b)$. The classical Runge Kutta monotonicity theory \cite{Higueras_2004_SSP,Kraaijevanger_1991_contractivity} 
states that an irreducible base scheme has a nonzero radius of absolute monotonicity if and only if
\begin{equation}
\label{eqn:incidence-base-scheme}
Inc\left( \begin{bmatrix} A & 0 \\ b^T & 0 \end{bmatrix}^2 \right) \le Inc\left( \begin{bmatrix} A & 0 \\ b^T & 0 \end{bmatrix} \right)\,,
\end{equation}
where $Inc$ denotes the incidence matrix 
(i.e., a matrix with entries equal to one or zero for the non-zero and
zero entries of the original matrix, respectively). 

The Butcher tableau of the resulting multirate scheme is
\[
\renewcommand{\arraystretch}{1.5}
\begin{array}{cccc|c}
\frac{1}{M} A     &          0                   & \cdots & 0 & A^{\{\f,\s,1\}}\\
\frac{1}{M} \mathbf{1} b\,^T & \frac{1}{M} A        & \cdots & 0 & A^{\{\f,\s,2\}}  \\
\vdots                     &                             & \ddots &   & \vdots \\
\frac{1}{M} \mathbf{1} b\,^T & \frac{1}{M} \mathbf{1} b\,^T   & \ldots & \frac{1}{M} A & A^{\{\f,\s,M\}} \\
\hline
\frac{1}{M} A^{\{\s,\f,1\}}  & \frac{1}{M} A^{\{\s,\f,2\}}   & \ldots & \frac{1}{M}A^{\{\s,\f,M\}}  & A \\
\hline
\frac{1}{M} b\,^T & \frac{1}{M} b\,^T & \ldots & \frac{1}{M} b\,^T & b^T
\end{array}
\renewcommand{\arraystretch}{1.0}
\]
and the matrix \eqref{eqn:mr-Ahat} is
\begin{equation}
\renewcommand{\arraystretch}{1.5}
\widehat{\mathbf{A}} = 
\label{eqn:mr-Ahat-extended}
\left[
\begin{array}{ccccc|cc}
 A     &          0                   & \cdots & 0 & 0 & A^{\{\f,\s,1\}} & 0 \\
 \mathbf{1} b\,^T &  A        & \cdots & 0 &0 &  A^{\{\f,\s,2\}} & 0 \\
\vdots                     &                             & \ddots &  & & \vdots & \\
 \mathbf{1} b\,^T &  \mathbf{1} b\,^T   & \ldots &  A &0 &  A^{\{\f,\s,M\}} & 0 \\
 b\,^T &  b\,^T & \ldots &  b\,^T &0 &  0 & 0 \\
\hline
 A^{\{\s,\f,1\}}  &  A^{\{\s,\f,2\}}   & \ldots & A^{\{\s,\f,M\}}  &0 &  A & 0 \\
 b\,^T &  b\,^T & \ldots &  b\,^T &0 &  b^T & 0
\end{array}
\renewcommand{\arraystretch}{1.0}
\right].
\end{equation}
The extra stages added to obtain \eqref{eqn:mr-Ahat-extended} from  \eqref{eqn:mr-Ahat} do not impact the final solution, 
therefore the underlying numerical solution is not changed.
Denote the upper left block of \eqref{eqn:mr-Ahat-extended} by $A_M$.  Note that \eqref{eqn:incidence-base-scheme}
implies that $Inc(A_M^2) \le Inc(A_M)$ as $A_M$ represents $M$ concatenated steps of the base method.
 
By similar arguments as in the classical theory \cite{Higueras_2004_SSP,Kraaijevanger_1991_contractivity}, the multirate scheme is absolutely monotonic if the incidence of the extended matrix satisfies
 \begin{equation}
 \label{eqn:am-incidence-inequality}
 Inc(\widehat{\mathbf{A}}^2) \le Inc(\widehat{\mathbf{A}})\,. 
\end{equation}
\begin{theorem}[Conditions for absolutely monotonic telescopic multirate GARK schemes]
The multirate GARK schemes with the same basic scheme for fast and slow components is absolutely monotonic, if the following conditions hold:
\begin{subequations}
\label{eqn:am-incidence-conditions}
\begin{eqnarray}
\label{eqn:am-incidence-11}
Inc\left( A_M^2 + \left(  A^{\{\f,\s,i\}} A^{\{\s,\f,j\}}  \right)_{i,j=1,\dots,M}\right) &\le& Inc\left( A_M  \right) \,,   \\
\label{eqn:am-incidence-22}
Inc\left(  \begin{bmatrix} A & 0 \\ b^T & 0 \end{bmatrix}^2 + 
\begin{bmatrix}  \sum_{\lambda=1}^M  A^{\{\s,\f,\lambda\}} A^{\{\f,\s,\lambda\}} & 0 \\ \sum_{\lambda=1}^M  b^T A^{\{\f,\s,\lambda\}}  & 0 \end{bmatrix} \right) &\le& Inc\left( \begin{bmatrix} A & 0 \\ b^T & 0 \end{bmatrix} \right)\,, \\
\label{eqn:am-incidence-12}
Inc\left(\sum_{\lambda=1}^{i-1} \one  b^T A^{\{\f,\s,\lambda\}}  +   A \, A^{\{\f,\s,i\}} +  A^{\{\f,\s,i\}} \, A \right) &\le& Inc\left( A^{\{\f,\s,i\}} \right) \,,   \\
\label{eqn:am-incidence-21a}
Inc\left(\sum_{\lambda=j+1}^{M} A^{\{\s,\f,\lambda\}} \one  b^T  +  A^{\{\s,\f,j\}}\, A + A\, A^{\{\s,\f,j\}}\right) &\le&  Inc\left(A^{\{\s,\f,j\}}\right)\,, \\
\label{eqn:am-incidence-21b}
Inc\left((M-j)  b^T  +  b^T\, A + b^T\, A^{\{\s,\f,j\}}\right) &\le& Inc\left(b^T\right)\,,
\end{eqnarray}
for all $i,j=1,\dots,M$.
\end{subequations}
\end{theorem}
\begin{proof}
The square of matrix \eqref{eqn:mr-Ahat-extended} is
\[
\renewcommand{\arraystretch}{1.5}
\widehat{\mathbf{A}}^2 = \begin{bmatrix} \widehat{\mathbf{A}}^2_{1,1} & \widehat{\mathbf{A}}^2_{1,2} \\ \widehat{\mathbf{A}}^2_{2,1} & \widehat{\mathbf{A}}^2_{2,2} \end{bmatrix}\,,
\renewcommand{\arraystretch}{1.0}
\]
with the following blocks:
\begin{eqnarray*}
\renewcommand{\arraystretch}{1.5}
\widehat{\mathbf{A}}^2_{1,1} &=&  A_M^2 + \left(  A^{\{\f,\s,i\}} A^{\{\s,\f,j\}}  \right)_{i,j=1,\dots,M}\,, \\
\widehat{\mathbf{A}}^2_{2,2} &=& \begin{bmatrix} A & 0 \\ b^T & 0 \end{bmatrix}^2 + 
\begin{bmatrix}  \sum_{\lambda=1}^M  A^{\{\s,\f,\lambda\}} A^{\{\f,\s,\lambda\}} & 0 \\  \sum_{\lambda=1}^M  b^T A^{\{\f,\s,\lambda\}}  & 0 \end{bmatrix}\,, \\
\widehat{\mathbf{A}}^2_{1,2} &=& \left(  \begin{bmatrix}   \sum_{\lambda=1}^{i-1} \one  b^T A^{\{\f,\s,\lambda\}}  +  
A \, A^{\{\f,\s,i\}} +  A^{\{\f,\s,i\}} \, A ,& 0 \end{bmatrix} \right)_{i=1,\dots,M}\,, \\
\widehat{\mathbf{A}}^2_{2,1} &=& \left(  \begin{bmatrix} \sum_{\lambda=j+1}^{M} A^{\{\s,\f,\lambda\}} \one  b^T  +  A^{\{\s,\f,j\}}\, A + A\, A^{\{\s,\f,j\}} \\
(M-j)  b^T  +  b^T\, A + b^T\, A^{\{\s,\f,j\}}
 \end{bmatrix} \right)_{j=1,\dots,M}\,.
\renewcommand{\arraystretch}{1.0}
\end{eqnarray*}
Writing the incidence inequality \eqref{eqn:am-incidence-inequality} by blocks 
yields~\eqref{eqn:am-incidence-conditions}.
\end{proof}
\begin{remark}
\begin{remunerate}
\item A comparision of~\eqref{eqn:am-incidence-22}
with \eqref{eqn:incidence-base-scheme} reveals that the monotonicity of the base scheme is a necessary, but not sufficient condition for the
absolute monotonicity of the multirate version.
The coupling coefficients have to be chosen appropriately in order to preserve this property.
For example, since $A_M$ and $A_M^2$ are block lower triangular, a necessary condition for  \eqref{eqn:am-incidence-11} is that
$ A^{\{\f,\s,i\}} A^{\{\s,\f,j\}} = \mathbf{0}$ for $i > j$.
\item If all weights of the base method are nonzero then condition \eqref{eqn:am-incidence-21b} is automatically satisfied.
\item An interesting choice of coupling coefficients is to use only the first micro-step solution in the slow calculation
\[
A^{\{\s,\f,\lambda\}} = {0}\,, \quad \lambda = 2, \dots, M\,,
\]
and to include the slow term contribution only in the last micro-step
\[
A^{\{\f,\s,\lambda\}} = 0\,, ~~ \lambda = 1,\dots,M-1\,.
\]
In this case the incidence conditions \eqref{eqn:am-incidence-11}--\eqref{eqn:am-incidence-21a} take the much simpler form:
\begin{subequations}
\label{eqn:am-incidence-conditions-simple}
\begin{eqnarray}
\label{eqn:am-incidence-11-simple}
Inc\left( A^2 + A^{\{\f,\s,M\}} A^{\{\s,\f,1\}}  \right) &\le& Inc\left( A \right)\, ,    \\
\label{eqn:am-incidence-22-simple}
Inc\left( b^T A + b^T A^{\{\f,\s,M\}}   \right) &\le& Inc\left( b^T \right)\, , \\
\label{eqn:am-incidence-12-simple}
Inc\left(\ A \, A^{\{\f,\s,M\}} +  A^{\{\f,\s,M\}} \, A \right) &\le& Inc\left( A^{\{\f,\s,M\}} \right) \, ,  \\
\label{eqn:am-incidence-21a-simple}
Inc\left( A^{\{\s,\f,1\}}\, A + A\, A^{\{\s,\f,1\}}\right) &\le&  Inc\left(A^{\{\s,\f,1\}}\right)\,. 
\end{eqnarray}
\end{subequations}
Condition \eqref{eqn:am-incidence-22-simple} is automatically satisfied if $b > 0$. Moreover, if the couplings are multiples of
the base scheme, $A^{\{\s,\f,1\}} = c_1\, A$ and $A^{\{\f,\s,M\}} = c_2\, A$, then \eqref{eqn:incidence-base-scheme} implies that all conditions \eqref{eqn:am-incidence-conditions-simple}
are satisfied.
\end{remunerate}
\end{remark}

\begin{example}[Monotonicity of an explicit multirate GARK scheme of order two]
The base for both the fast and the slow schemes is
the following explicit, order two, strong stability preserving method, with an absolute stability radius $\mathcal{R}=1$
\begin{equation}
\label{eqn:ssp2}
\renewcommand{\arraystretch}{1.25}
A =  \begin{bmatrix}  0 & 0 \\ 1 & 0 \end{bmatrix} \,, \quad 
b   =  \begin{bmatrix}  \frac{1}{2}\\ \frac{1}{2} \end{bmatrix} \,, \quad
c  =   \begin{bmatrix} 0\\ 1 \end{bmatrix}\,.
\end{equation}
 The general coupling conditions for a second order multirate scheme read:
\begin{eqnarray*}
\frac{M}{2} &=  & M\, b^{\{\s\}}\,^T A^{s,a} \one  = \sum_{\lambda=1}^M b^T \, A^{\{\s,\f,\lambda\}}\, \one\,, \\
 \frac{1}{2} & =  & b^{\{\f\}}\,^T A^{\{\f,\s\}} \one  = \frac{1}{M} \sum_{\lambda=1}^M b^T \, A^{\{\f,\s,\lambda\}}\, \one\,.
\end{eqnarray*}
We consider three different couplings.
\begin{itemize}
\item The coupling coefficients that respect the nonlinear stability decoupling condition are:
\[
A^{\{\s,\f\}} = \begin{bmatrix} 0 & 0 \\ M & 0  \end{bmatrix},
\quad
D = \begin{bmatrix} 0 & M \\ 0 & 0  \end{bmatrix} = A^{\{\s,\f\}}\,^T,
\]
\[
{\mathbf{A}}^{\{\s,\f\}}  = 
\begin{bmatrix}
\frac{1}{M} 
A^{\{\s,\f\}} & 0 & \cdots & 0  
\end{bmatrix}
=
\begin{bmatrix}
A^{\{\s,\s\}} & 0 & \cdots & 0  
\end{bmatrix}.
\]
\begin{eqnarray*}
\renewcommand{\arraystretch}{1.5}
\mathbf{A^{\{\f,\s\}}} := 
\begin{bmatrix}
\one {b}^{\{\s\}}\,^T - B^{\{\f\}}\,^{-1} A^{\{\s,\f\}}\,^T B^{\{\s\}}  \\
\one {b}^{\{\s\}}\,^T \\
\vdots \\
\one {b}^{\{\s\}}\,^T
\end{bmatrix}
=
\begin{bmatrix}
\frac{1}{2} &  \frac{1}{2}-M \\
\frac{1}{2} &  \frac{1}{2} \\
\vdots  & \vdots \\
\frac{1}{2} &  \frac{1}{2}
\end{bmatrix}
\renewcommand{\arraystretch}{1.0}
\end{eqnarray*}
Because of the negative coefficient this method is not absolutely stable.

\item Coupling the slow step with only the first fast step is achieved by
\begin{equation}
\label{eqn:slow-with-first-fast}
A^{\{\s,\f,1\}} = \begin{bmatrix} 0 & 0 \\ M & 0  \end{bmatrix} \,, \quad
A^{\{\s,\f,\lambda\}} = \mathbf{0}\,, ~~~ \lambda \ge 2\,.
\end{equation}
The second order condition can be fulfilled by taking
\[
A^{\{\f,\s,\lambda\}} = A\,, \quad \forall \, \lambda\,.
\]
%
%
For $M \ge 2$ we have $\mathcal{R}=0$, since \eqref{eqn:mr-Ahat} corresponds to an irreducible Runge Kutta scheme, and \eqref{eqn:am-incidence-12}
is not satisfied.

\item We now build a scheme with \eqref{eqn:slow-with-first-fast}, and which includes the slow terms only in the last micro-step
\[
A^{\{\f,\s,\lambda\}} = 0\,, ~~ \lambda = 1,\dots,M-1\,, \quad
A^{\{\f,\s,M\}} = M\, A\,.
\]
With this coupling the multirate scheme maintains the absolute stability radius of the base method for any $M$, as all conditions \eqref{eqn:am-incidence-conditions-simple} are satisfied.
\end{itemize}
\end{example}

We note that monotonicity conditions for several multirate and partitioned explicit Runge-Kutta schemes are  discussed by Hundsdorfer, Mozartova, and Savcenco in a recent report \cite{Hundsdorfer_2013_mr-monotonicity}.


\section{Conclusions and future work}\label{sec:conclusions}

This work develops multirate generalized additive Runge Kutta schemes, which exploit the different levels of activity within the partitions of the right-hand sides and/or components by using 
appropriate time steps for each of them. Multirate GARK schemes inherit the high level of flexibility from GARK schemes~\cite{SaGu13a}, which allow for different stage values 
as arguments of different components of the right hand side. Many well-known multirate Runge-Kutta schemes, such as the Kvaerno-Rentrop methods~\cite{Kvaerno_1999_MR-RK}, the multirate infinitesimal step methods~\cite{Knoth_1998_MRimex}, and methods based on dense output coupling, are particular particular members of the new family of methods. 

The paper develops the order conditions (up to order three) for the generalized additive multirate schemes. 
We extend the GARK algebraic stability and monotonicity analysis~\cite{SaGu13a} to the new multirate family, and show that these properties are inherited from the base schemes provided that some coupling conditions hold. 

Future work will construct multirate GARK methods tailored to special 
applications in, for example, circuit design, vehicle system dynamics, or air quality modeling, and will extend the
new family of schemes to differential-algebraic equations.

\section*{Acknowledgements}
The work of A. Sandu has been supported in part by NSF through awards NSF
OCI--8670904397, NSF CCF--0916493, NSF DMS--0915047, NSF CMMI--1130667, 
NSF CCF--1218454, AFOSR FA9550--12--1--0293--DEF, AFOSR 12-2640-06,
and by the Computational Science Laboratory at Virginia Tech.

The work of M. G\"unther has been supported in part by BMBF through grant 
03MS648E.
\bibliographystyle{siam}

\begin{thebibliography}{10}

\bibitem{Higueras_2004_SSP}
{\sc I.~Higueras}, {\em On strong stability preserving time discretization
  methods}, Journal of Scientific Computing, 21 (2004), pp.~193--223.

\bibitem{Higueras_2005_monotonicity}
\leavevmode\vrule height 2pt depth -1.6pt width 23pt, {\em Monotonicity for
  {Runge-Kutta} methods: inner product norms}, Journal of Scientific Computing,
  24 (2005), pp.~97--117.

\bibitem{Higueras_2006_SSP-ARK}
\leavevmode\vrule height 2pt depth -1.6pt width 23pt, {\em Strong stability for
  additive {Runge-Kutta} methods}, SIAM Journal on Numerical Analysis, 44
  (2006), pp.~1735--1758.

\bibitem{Hundsdorfer_2013_mr-monotonicity}
{\sc W.~Hundsdorfer, A.~Mozartova, and V.~Savcenco}, {\em Monotonicity
  conditions for multirate and partitioned explicit {Runge-Kutta} schemes}.
\newblock CWI report, unpublished, 2013.

\bibitem{KeCa2003}
{\sc Ch. Kennedy and M.~Carpenter}, {\em Additive runge-kutta schemes for
  convection-diffusion-reaction equations}, Applied Numerical Mathematics, 44
  (2003), pp.~139--181.

\bibitem{Knoth_1998_MRimex}
{\sc O.~Knoth and R.~Wolke}, {\em Implicit-explicit {Runge-Kutta} methods for
  computing atmospheric reactive flows}, Applied Numerical Mathematics, 28
  (1998), pp.~327--341.

\bibitem{Kraaijevanger_1991_contractivity}
{\sc J.F.B.M. Kraaijevanger}, {\em Contractivity of {Runge-Kutta} methods},
  {BIT Numerical Mathematics}, 31 (1991), pp.~482--528.

\bibitem{Kvaerno_1999_MR-RK}
{\sc A.~Kvaerno and P.~Rentrop}, {\em Low order multirate {R}unge-{K}utta
  methods in electric circuit simulation}, Preprint 99/1, IWRMM, University of
  Karlsruhe, 1999.

\bibitem{Knoth_2010_MRimex}
{\sc S.~Matrin and O.~Knoth}, {\em Multirate implicit-explicit time integration
  schemes in atmospheric modelling}, in AIP conference proceedings, vol.~1281,
  CNAAM 2010: International Conference of Numerical Analysis and Applied
  Mathematics, 2010, p.~1831.

\bibitem{SaGu13a}
{\sc A.~Sandu and M.~G\"unther}, {\em A class of generalized additive
  {Runge-Kutta} methods}.
\newblock Work in progress.

\bibitem{Savcenco_2007TR}
{\sc V.~Savcenco}, {\em Construction of high-order multirate {R}osenbrock
  methods for stiff {ODE}s}, Tech. Report MAS-E0716, Centrum voor Wiskunde en
  Informatica, 2007.

\bibitem{Savcenco_2008}
\leavevmode\vrule height 2pt depth -1.6pt width 23pt, {\em Comparison of the
  asymptotic stability properties for two multirate strategies}, Journal of
  Computational and Applied Mathematics, 220 (2008), pp.~508--524.

\bibitem{Savcenco_2009}
\leavevmode\vrule height 2pt depth -1.6pt width 23pt, {\em Construction of a
  multirate rodas method for stiff odes}, Journal of Computational and Applied
  Mathematics, 225 (2009), pp.~323 -- 337.

\bibitem{Savcenco_2005}
{\sc V.~Savcenco, W.~Hundsdorfer, and J.G. Verwer}, {\em A multirate time
  stepping strategy for parabolic {PDE}}, Tech. Report MAS-E0516, {Centrum voor
  Wiskundeen Informatica}, 2005.

\bibitem{Savcenco_2007_stability}
{\sc V.~Savcenko}, {\em Comparison of the asymptotic stability properties for
  two multirate strategies}, Tech. Report MAS-R0705, CWI, Amsterdam, 2007.

\bibitem{Savcenco_2007_MRstrategy}
{\sc V.~Savcenko, W.~Hundsdorfer, and J.G. Verwer}, {\em A multirate time
  stepping strategy for stiff odes}, BIT, 47 (2007), pp.~137--155.

\bibitem{Schlegel_2009_MR-advection}
{\sc M.~Schlegel, O.~Knoth, M.~Arnold, and R.~Wolke}, {\em {Multirate
  {R}unge-{K}utta schemes for advection equations}}, Journal of Computational
  and Applied Mathematics, 226 (2009), pp.~345--357.

\bibitem{Schlegel_2010_MR-imex}
\leavevmode\vrule height 2pt depth -1.6pt width 23pt, {\em Multirate
  implicit-explicit time integration schemes in atmospheric modelling}, in AIP
  Conference Proceedings, vol.~1281, 2010, pp.~1831--1834.

\bibitem{Schlegel_2011_MRimplementation}
\leavevmode\vrule height 2pt depth -1.6pt width 23pt, {\em Implementation of
  splitting methods for air pollution modeling}, Geoscientific Model
  Development Discussions, 4 (2011), pp.~2937--2972.

\bibitem{Schlegel_2012_multiscale}
\leavevmode\vrule height 2pt depth -1.6pt width 23pt, {\em Numerical solution
  of multiscale problems in atmospheric modeling}, Applied Numerical
  Mathematics, 62 (2012), pp.~1531--1543.

\bibitem{Wensch_2009_atmospheric}
{\sc J.~Wensch, O.~Knoth, and A.~Galant}, {\em Multirate infinitesimal step
  methods for atmospheric flow simulation}, BIT Numerical Mathematics, 49
  (2009), pp.~449--473.

\end{thebibliography}

\end{document}